\documentclass[11pt]{amsart} 
\usepackage{amsmath, amsthm, amscd, amsfonts, amssymb, enumerate, verbatim, newlfont, calc, graphicx, color}
\usepackage[bookmarksnumbered, colorlinks, plainpages]{hyperref}
\usepackage{tikz}
 \usepackage{doi}
\usepackage{cancel}
\newtheorem{theorem}{Theorem}[section] 
\newtheorem{question}[theorem]{Question}

\newtheorem{lemma}[theorem]{Lemma}
\newtheorem{proposition}[theorem]{Proposition}

\newtheorem{corollary}[theorem]{Corollary}
\newtheorem{definition}[theorem]{Definition}

\newtheorem{remark}[theorem]{Remark}

\newtheorem{example}[theorem]{Example}

\makeatletter
\def\subsection{\@startsection{subsection}{2}{\z@}%
  {-3.5ex\@plus -1ex\@minus -.2ex}%
  {2ex\@plus .2ex}%
  {\normalfont\bfseries}}
\makeatother

\def\Ass{\operatorname{Ass}}
\def\Min{\operatorname{Min}}
\def\depth{\operatorname{depth}}

\usepackage{mathtools}
\usepackage{pstricks}
\usepackage{epsfig}
\usepackage{pst-grad}
\usepackage{pst-plot}
\usepackage{subfig}
\usepackage{academicons}
\usepackage{xcolor}
 \usepackage{academicons}
 \definecolor{orcidlogocol}{HTML}{A6CE39}

\begin{document}

\author[
M. Cimpoea\c{s}, M. Nasernejad and A. A. Qureshi
]{
Mircea Cimpoea\c{s}$^{1,2}$,
Mehrdad Nasernejad$^{3,4,*}$, and
Ayesha Asloob Qureshi$^{5}$
}

\title[On the existence of  the maximal ideal]{On the existence of  the maximal ideal in  the set of associated primes of monomial ideals}

\thanks{$^*$Corresponding author}

\thanks{E-mail addresses:  mircea.cimpoeas@imar.ro, m$\_$nasernejad@yahoo.com and    aqureshi@sabanciuniv.edu}  
\maketitle

\begin{center}
{\it
$^{1}$National University of Science and Technology Politehnica Bucharest,\\
Faculty of Applied Sciences, Bucharest 060042, Romania\\
$^{2}$
Simion Stoilow Institute of Mathematics, Research Unit 5, P.O. Box 1-764, Bucharest 014700, Romania\\
$^{3}$Univ. Artois, UR 2462, Laboratoire de Math\'{e}matique de  Lens (LML), \\  F-62300 Lens, France \\ 
$^{4}$Universit\'e  Caen Normandie, ENSICAEN, CNRS, Normandie Univ, GREYC UMR  6072, F-14000 Caen,  France\\
$^{5}$Sabanc\i\;University, Faculty of Engineering and Natural Sciences, \\
Orta Mahalle, Tuzla 34956, Istanbul, Turkey 
}
\end{center}

\vspace{0.4cm}

\begin{abstract}
In this paper, we investigate the behavior of associated primes of powers of monomial ideals, with particular emphasis on the presence of the maximal ideal and the phenomenon of fluctuation. We first establish several criteria for determining when the maximal ideal belongs to the set of associated primes of powers of monomial ideals in $K[x,y,z]$. We then provide examples illustrating the appearance and disappearance of the maximal ideal among the associated primes of successive powers. Furthermore, we prove that for every $n\geq 3$, there exist infinitely many monomial ideals in $K[x_1,\ldots,x_n]$ whose powers exhibit fluctuations in their sets of associated primes. Finally, we construct infinitely many examples of nearly normally torsion-free (respectively, co-nearly normally torsion-free) monomial ideals that fail to satisfy the persistence (respectively, copersistence) property.
\end{abstract}

\vspace{0.4cm}


\section{Introduction and Overview}

Suppose  that $R$ is  a commutative  Noetherian ring and $I$ is an ideal of $R$. A prime ideal $\mathfrak{p}\subset  R$ is an {\it associated prime} of $I$ if there exists an element $v$ in $R$ such that $\mathfrak{p}=(I:_R v)$, where $(I:_R v)=\{r\in R \mid  rv\in I\}$. The  {\it set of associated primes} of $I$, denoted by  $\mathrm{Ass}_R(R/I)$, is the set of all prime ideals associated to  $I$.  Brodmann \cite{BR} showed that the sequence $\{\mathrm{Ass}_R(R/I^k)\}_{k \geq 1}$ of associated prime ideals is stationary  for large $k$. That is, there exists a positive integer $k_0$ such that $\mathrm{Ass}_R(R/I^k)=\mathrm{Ass}_R(R/I^{k_0})$ for all $k\geq k_0$. The  minimal such $k_0$ is called the {\it index of stability}   of  $I$ and $\mathrm{Ass}_R(R/I^{k_0})$ is called the {\it stable set }  of associated prime ideals of  $I$, which is denoted by $\mathrm{Ass}^{\infty }(I).$ In addition, suppose that  $I=\bigcap_{i=1}^r\mathfrak{q}_i$ is   a \textit{minimal primary decomposition} of $I$ with $\sqrt{\mathfrak{q}_i}=\mathfrak{p}_i$ for all $i$,  that is, $\mathfrak{q}_i$ is a $\mathfrak{p}_i$-primary ideal, 
  $\sqrt{\mathfrak{q}_i}\neq \sqrt{\mathfrak{q}_j}$ for any $i\neq j$, and $I\neq \bigcap_{i\neq j}\mathfrak{q}_i$ for all $j =1, \ldots, r$. Then the  set of  associated primes of  $I$ is equal to $\{\mathfrak{p}_1, \ldots, \mathfrak{p}_r\}$, refer to  \cite[Theorem 4.5]{Atiyah}. 
 
 Suppose  now  that $I$ is a monomial  ideal  in a polynomial ring $R=K[x_1,\ldots,x_n]$ over a field $K$, $\mathfrak{m}=(x_1, \ldots, x_n)$ is  the  maximal ideal of $R$, and $x_1,\ldots,x_n$ are indeterminates.  One of the most interesting  questions in this subject is the existence of the  maximal ideal in the set of associated primes. Generally,  little is known  about  this subject in the literature. However, in recent studies, by utilizing  combinatorial approaches, some  papers have been published  to give answers to this question  for (square-free)  monomial ideals. 
   Chen  et  al., in \cite[Lemma 3.1]{CMS}, proved that if $G$  is a cycle of length 
 $2k+1$ and $I$ is the edge ideal of $G$, then $\mathrm{Ass}(R/I^n) = \mathrm{Min}(R/I)\cup \{\mathfrak{m}\}$ when  $n \geq  k + 1$. 
 Similarly, it has been shown in \cite[Proposition 3.6]{NKA} that if $G$  is an odd cycle and $J$ is the cover ideal of  $G$, then  $\mathfrak{m}\in \mathrm{Ass}(J^s)$  for all $s\geq 2$.  Besides,  in \cite[Lemma 3.5]{NQT}, certain results have been released regarding the existence of maximal ideals in the edge and cover ideals of cones of graphs.   Moreover,   Herzog et  al., in \cite[Corollaries 4.5 and  5.5]{HRV}, studied  the existence of the maximal ideal in the cases of   transversal polymatroidal ideals and    ideals of Veronese type, respectively.  Furthermore, H\`a and Morey,  in \cite[Corollary 3.6]{HM},  presented  a lower bound for the least power $t$ such that $I^t$  has embedded primes. 
 The main motivation of this paper is to investigate the presence of the maximal ideal $(x,y,z)$ among the associated primes of powers of monomial ideals in $K[x,y,z].$  
 
  The paper is organized as follows. In Section \ref{Section2},  we present the basic definitions, results, and preliminary facts that are essential for the developments in the subsequent sections of this paper.
 
 In  Section  \ref{Section3}, we establish criteria for the existence of  the maximal ideal $(x,y,z)$  in the set of associated primes of powers of   monomial ideals in $K[x,y,z]$, as discussed in  Propositions \ref{Pro.Detecting-K[x,y,z]} and \ref{Pro.Not.maximal.1}. 
  On the other hand, it is well known that $\mathfrak{m}\in \mathrm{Ass}(R/I)$  if and only if $\mathrm{depth}(R/I)=0$, see \cite[Exercise  2.2.11]{V1}.  Hence, a bunch  of papers have been written in the language of depth notion, more information can be found in \cite{BHH, HNTT,  HH2, MST, RS}. From this perspective, we complete this section by giving some  criteria for the existence of the maximal ideal in the set of associated primes, 
 refer to Corollaries \ref{p1}, \ref{c3} and \ref{C4}. 
 
 In Section \ref{Section4}, we continue to investigate the presence of the maximal ideal $(x,y,z)$ in the set of associated primes of powers of monomial ideals in $K[x,y,z]$. Through a series of lemmas and five illustrative examples, we analyze the behavior of the maximal ideal for different powers of the ideal, determining precisely when it belongs to or disappears from the sets of associated primes. Each case corresponds to a distinct configuration of the minimal primary decomposition and is supported by explicit constructions and detailed proofs. See Examples \ref{Ex.Ass.PandN}, \ref{Ex.Ass.NandP}, 
 \ref{Ex.Ass.NandP.3}, \ref{Ex.Ass.PandN.2}, and \ref{Ex.Ass.PandN.3}. 

 Finally, Section \ref{Section5}   is devoted to studying the  phenomenon of fluctuation in associated primes of powers of monomial ideals. 
 Let $I\subset R=K[x_1, \ldots, x_n]$ be a monomial ideal. One of the most  intriguing questions in this area is the following:

\medskip
\noindent\textbf{Question.}
Does there exist a monomial prime ideal $\mathfrak{p}$ in $R$ and positive
integers $a < b < c$ such that at least one of the following cases holds?

\begin{itemize}
    \item[(i)] $\mathfrak{p} \in \mathrm{Ass}(R/I^a)$,
    $\mathfrak{p} \notin \mathrm{Ass}(R/I^b)$, and
    $\mathfrak{p} \in \mathrm{Ass}(R/I^c)$.

    \item[(ii)] $\mathfrak{p} \notin \mathrm{Ass}(R/I^a)$,
    $\mathfrak{p} \in \mathrm{Ass}(R/I^b)$, and
    $\mathfrak{p} \notin \mathrm{Ass}(R/I^c)$.
\end{itemize}
This phenomenon is known as \textit{fluctuation in  associated primes of powers}. Several interesting examples of such fluctuations can be
found in \cite{BHH,HH2,RS}.
For example, let $\mathfrak{m}=(x_1, \ldots, x_6)$ and consider the following monomial ideal 
  $$I=(x_1^6, x_1^5x_2, x_1x_2^5, x_2^6, x_1^4x_2^4x_3, x_1^4x_2^4x_4, x_1^4x_5^2x_6^3, x_2^4x_5^3x_6^2),$$ 
 in $R=K[x_1, x_2, x_3, x_4, x_5, x_6]$ from \cite[p.~549]{HH2}. 
  Using  \textit{Macaulay2} \cite{GS}, we have  $\mathfrak{m}\in \mathrm{Ass}_R(R/I)$, $\mathfrak{m}\notin \mathrm{Ass}_R(R/I^2)$,  
 $\mathfrak{m}\in \mathrm{Ass}_R(R/I^3)$, and  $\mathfrak{m}\notin \mathrm{Ass}_R(R/I^4)$.  
 Another interesting example can be found in \cite[Exercise 7.7.25]{V1}. Indeed, consider  the following monomial ideal 
$$I=(x_1x_2^2x_3,\; x_2x_3^2x_4,\; x_3x_4^2x_5,\; x_4x_5^2x_1,\; x_5x_1^2x_2) \subset K[x_1, x_2, x_3, x_4, x_5],$$
 and  $\mathfrak{m} = (x_1, \dots, x_5)$.  Using  \textit{Macaulay2} \cite{GS}, we get  $\mathfrak{m}\in \mathrm{Ass}_R(R/I)$, $\mathfrak{m}\notin \mathrm{Ass}_R(R/I^2)$,   $\mathfrak{m}\notin \mathrm{Ass}_R(R/I^3)$, and  $\mathfrak{m}\in \mathrm{Ass}_R(R/I^4)$.
Furthermore, by virtue of    \cite[Theorem 9]{RT} (respectively, \cite[Theorem 3.18]{BNT}), any square-free monomial ideal in 
 $K[x_1,x_2,x_3,x_4]$ (respectively, $K[x_1,x_2,x_3,x_4,x_5]$) has the strong persistence property, and so has the persistence property (for example see  the proof of \cite[Proposition 2.9]{N1}). This implies that no fluctuations occur for  square-free monomial ideals in $K[x_1,x_2,x_3,x_4,x_5]$. 
 An example of a square-free monomial ideal in $R=K[x_1,\ldots,x_7]$ (respectively, in $R=K[x_1,\ldots,x_{12}]$) exhibiting fluctuation in
the associated primes of its powers can be found in \cite[Example 2.9]{MS} (respectively, in  \cite{Kaiser2014} and \cite[p.~3788]{ANR2}).
 Along these lines, we prove that if $R=K[x_1,\ldots,x_n]$ with $n\geq 3$, then there exist infinitely many monomial ideals in $R$
whose powers exhibit fluctuation in their sets of associated primes; see Theorem~\ref{Th.Fluctuation}.

Next, we verify that if $R=K[x_1,\ldots,x_n]$ with $n\geq 3$, then there exist infinitely many monomial ideals in $R$
 that are nearly normally torsion-free (respectively, co-nearly normally torsion-free) but do not satisfy the
persistence property (respectively, copersistence property), refer to Proposition  \ref{Pro.NNTF} (respectively, Proposition \ref{Pro.co-NNTF}).

   
Throughout this paper,  $\mathcal{G}(I)$ denotes the unique minimal set of monomial generators of a monomial ideal $I\subset R=K[x_1, \ldots, x_n]$, where 
$R=K[x_1, \ldots, x_n]$ is  a polynomial  ring over a field $K$. Furthermore, the {\em support} of a monomial $u\in R$, denoted by $\mathrm{supp}(u)$, is the set of variables that divide $u$; in particular, we consider $\mathrm{supp}(1)=\emptyset$.  Moreover, for a monomial ideal $I$, we set $\mathrm{supp}(I)=\bigcup_{u \in \mathcal{G}(I)}\mathrm{supp}(u)$.


\section{Preliminaries} \label{Section2}

In this section, we gather essential results, definitions, and facts that will be used throughout the rest of this paper. 
We begin by reviewing the following theorem.


\begin{theorem} (\cite[Theorem 5.2] {KHN2})\label{5.2KHN}
Let $I$ be a monomial ideal of $R$ with $\mathcal{G}(I) =\{u_1,\ldots,u_m\}$. Also, assume that there exists a monomial $h=x_{j_1}^{b_1}\cdots x_{j_s}^{b_s}$ such that $h \mid  u_i$ for all $i=1,\ldots,m$. Set $J:=(u_1/h,\ldots,u_m/h)$. Then  we have      $\mathrm{Ass}_R(R/I)=\mathrm{Ass}_R(R/J)\cup\{ (x_{j_1}),\ldots,(x_{j_s})\}.$
\end{theorem}


\begin{definition} (\cite[Definition 6.1.1]{MRS}) \label{Def. parameter ideal}
\em{
Set  $R = A[X_1,\ldots,X_d]$. A \textit{parameter ideal}  in $R$ is an ideal of the form $(X^{a_1}_1 , \ldots, X^{a_d}_d)R$ 
 with $a_1,  \ldots, a_d \geq 1$.  If  $f = X^{n_1}_1 \cdots  X^{n_d}_d$ with $n_i\geq 0$ for all $i=1, \ldots, d$, then we set
$\mathrm{P}_R( f) := (X^{n_1+1}_1 , \ldots, X^{n_d+1}_d)R.$  
}
\end{definition}


\begin{definition} (\cite[Definition 6.2.1]{MRS}) \label{Def. Corner elements}
\em{
Let $I$ be a monomial ideal in $R=K[x_1, \ldots, x_n]$. A monomial $f\in R$  is called an \textit{$I$-corner element} if $f\notin I$  and 
$x_1f, \ldots, x_nf \in I$.  The set of  $I$-corner elements of $I$  is denoted by $\mathrm{C}_R(I)$.
In particular,  $f$ is an $I$-corner element  of a monomial ideal $I \subset R=K[x_1, \ldots, x_n]$ if and only if $(I:f)=\mathfrak{m}\in \mathrm{Ass}(R/I)$, 
where $\mathfrak{m}=(x_1, \ldots, x_n)$.  
}
\end{definition}


\begin{proposition} (\cite[Proposition 6.2.17]{MRS}) \label{Pro.Parametric}
 Set $R = A[X_1, \ldots, X_d]$ and $\mathfrak{X} = (X_1, \ldots, X_d)R$. Let $J$ be a monomial ideal of $R$ with irredundant $m$-irreducible
  decomposition  $J =\bigcap_{i=1}^k J_i$ with $k \geq 1$.  Assume that the ideals $J_i$ are ordered so that 
  $\mathrm{m}$-$\mathrm{rad}(J_i) =\mathfrak{X}$    if and only if  $1\leq i \leq n$.  Then $\mathrm{C}_R(J)=\bigcup_{i=1}^n \mathrm{C}_R(J_i)$. In other words, if
    $J_i =\mathrm{P}_R(g_i)$ for all  $i=1, \ldots, n$, then the 
  distinct $J$-corner elements are $g_1, \ldots, g_n$. 
\end{proposition}




\begin{proposition}(\cite[Proposition 2.2]{NR}) \label{K[x,y].1}
Let $I$ be a monomial ideal in the polynomial ring $R=K[x,y]$ with
 $\mathcal{G}(I)=\{u_1,\ldots,u_m\}.$ If $m\geq 2$, then
$\mathfrak m=(x,y)\in\operatorname{Ass}_R(R/I).$
\end{proposition}



\section{Detecting  the presence of the maximal  ideal $(x,y,z)$ in the set of associated primes} \label{Section3}
In  this section,  we establish criteria for the existence of  the maximal ideal $(x,y,z)$  in the set of associated primes of powers of 
 monomial ideals in $K[x,y,z]$.  To achieve this, we first recall the following theorem, which is needed to formulate Lemma \ref{Lem.Detecting-K[x,y,z]}. We state it here for convenience.

\begin{theorem} \label{Th.Detecting-General} (\cite[Theorem 2.3]{NR})
Suppose that  $I$ is a monomial ideal in a polynomial ring $R=K[x_1, \ldots, x_n]$, $\mathfrak{m}=(x_1, \ldots, x_n)$, and 
  $$\mathcal{G}(I)=\{x^{r_{1,1}}_1\cdots x^{r_{1,n}}_n, \ldots, x^{r_{k,1}}_1\cdots x^{r_{k,n}}_n\},$$ with $k\geq n$.  Then  $\mathfrak{m}\in \mathrm{Ass}_R(R/I)$ if and only if there exist  distinct  integers $i_1, \ldots, i_n \in \{1, \ldots, k\}$ such that  the following conditions hold:
\begin{itemize}
\item[(i)] $r_{i_t,j}<r_{i_j,j}$ for each  $t\neq j$, where $j=1, \ldots, n$;
\item[(ii)]  for each  $i\in\{1, \ldots, k\}\setminus \{i_1, \ldots, i_n\}$, there exists a positive integer $1\leq t \leq n$ such that $r_{i,t}\geq r_{i_t,t}$.
\end{itemize}
\end{theorem}

Particularly, after recalling the definitions of  $\mathrm{P}_R(h)$ and $\mathrm{C}_R(I)$  from  Definitions \ref{Def. parameter ideal} and \ref{Def. Corner elements},   we can deduce from the proof of  Theorem \ref{Th.Detecting-General} and  Proposition \ref{Pro.Parametric} that $h={x_1}^{r_{{i_1},1}-1} \cdots {x_n}^{r_{{i_n},n}-1}\in \mathrm{C}_R(I)$ and  additionally  $\mathrm{P}_R(h)= ({x_1}^{r_{{i_1},1}}, \ldots,  {x_n}^{r_{{i_n},n}})$ appears in the irredundant irreducible decomposition of $I$.


The next lemma is an immediate consequence of Theorem \ref{Th.Detecting-General}, which will be used in the proof of 
Proposition \ref{Pro.Detecting-K[x,y,z]}.

\begin{lemma}\label{Lem.Detecting-K[x,y,z]} 
Let $I \subset R=K[x,y,z]$ be a monomial ideal,  $\mathfrak{m}=(x,y,z)$, $k\geq 3$, and  
$\mathcal{G}(I)=\{x^{r_{1,1}} y^{r_{1,2}} z^{r_{1,3}},  \ldots, x^{r_{k,1}} y^{r_{k,2}} z^{r_{k,3}}\}.$
 Then $\mathfrak{m}\in \mathrm{Ass}_R(R/I)$ if and only if there exist distinct integers $i_1, i_2, i_3 \in \{1, \ldots, k\}$ such that the following conditions hold:
\begin{itemize}
\item[(i)] $r_{i_t,j}<r_{i_j,j}$ for each  $t\neq j$, where $j=1, 2,3$;
\item[(ii)]  for each  $i\in\{1, \ldots, k\}\setminus \{i_1, i_2, i_3\}$, there exists a positive integer $1\leq t \leq 3$ such that $r_{i,t}\geq r_{i_t,t}$.
\end{itemize}
In particular, we have  $h=x^{r_{{i_1},1}-1} y^{r_{{i_2},2}-1} z^{r_{{i_3},3}-1} \in \mathrm{C}_R(I)$ and 
$\mathrm{P}_R(h)= (x^{r_{{i_1},1}}, y^{r_{{i_2},2}}, z^{r_{{i_3},3}})$ appears in the irredundant irreducible decomposition of $I$. 
\end{lemma}


 The next proposition  provides our first criterion.

\begin{proposition}\label{Pro.Detecting-K[x,y,z]} 
Let $I \subset R=K[x,y,z]$ be a monomial ideal,  $\mathfrak{m}=(x,y,z)$, $k\geq 3$, and  
 $\mathcal{G}(I)=\{x^{r_{1,1}} y^{r_{1,2}} z^{r_{1,3}},  \ldots, x^{r_{k,1}} y^{r_{k,2}} z^{r_{k,3}}\}.$
 Let there exist  some $1\leq s \leq k-2$ such that
 \begin{itemize}
 \item   $r_{1,1} \geq r_{2,1} \geq \cdots \geq r_{{i_s},1} > r_{{i_s+1}, 1}$ and    $r_{{i_s},1} > r_{{i_s+2}, 1},$ 
 \item   $r_{i_s,2} < r_{{i_s+1},2} > r_{{i_s+2},2},$   $r_{{i_s+1},2} > r_{i_{s},2}$  and   $r_{{i_s+1},2} > r_{{i_s+2},2}$, 
 \item    $r_{{i_s+1},3} < r_{{i_s+2},3} \leq \cdots \leq r_{k,3}$ and  $r_{i_s,3} < r_{{i_s+2},3}$. 
 \end{itemize}
 Then $\mathfrak{m} \in  \mathrm{Ass}(R/I)$.  In particular,  $h=x^{r_{i_s,1}-1} y^{r_{{i_s+1},2}-1} z^{r_{{i_s+2},3}-1}\in \mathrm{C}_R(I)$ and  $\mathrm{P}_R(h)=(x^{r_{i_s,1}}, y^{r_{{i_s+1},2}}, z^{r_{{i_s+2},3}})$  
 appears in the irredundant irreducible decomposition of $I$. 
\end{proposition}

\begin{proof}
Since $r_{i_s,1} > r_{{i_s+1},1}$,  $r_{i_s,1} > r_{{i_s+2},1}$, $r_{{i_s+1},2} > r_{i_{s},2}$,  $r_{{i_s+1},2} > r_{{i_s+2},2}$,  
$r_{{i_s+2},3} > r_{{i_s+1},3}$, and  $r_{{i_s+2},3} > r_{i_s,3}$,  we can deduce that condition (i) in Lemma \ref{Lem.Detecting-K[x,y,z]} holds. 
 Let $1\leq \ell \leq i_s-1$ and $i_s+3 \leq p \leq k$. Thanks to  $r_{\ell,1} \geq r_{i_s,1}$  and $r_{p,3} \geq r_{{i_s+2},3}$, this shows that condition (ii) in Lemma \ref{Lem.Detecting-K[x,y,z]} holds as well. Consequently,   we get  $\mathfrak{m}\in \mathrm{Ass}(R/I)$.  Furthermore, 
  based on   Lemma \ref{Lem.Detecting-K[x,y,z]},  we get $h=x^{r_{i_s,1}-1} y^{r_{{i_s+1},2}-1} z^{r_{{i_s+2},3}-1}$ is an $I$-corner element, that is, $x^{r_{i_s,1}-1} y^{r_{{i_s+1},2}-1} z^{r_{{i_s+2},3}-1}\in \mathrm{C}_R(I)$; in particular,  $\mathrm{P}_R(h)=(x^{r_{i_s,1}}, y^{r_{{i_s+1},2}}, z^{r_{{i_s+2},3}})$
   appears in the irredundant irreducible decomposition of $I$.  This completes  the proof.
 \end{proof}


 We now present the following proposition, which yields our second criterion.

\begin{proposition} \label{Pro.Not.maximal.1}
   Suppose that   $I \subset R=K[x,y,z]$ is  a monomial ideal, $\mathfrak{m}=(x,y,z)$, and 
$\mathcal{G}(I)=\{x^{r_{1,1}} y^{r_{1,2}} z^{r_{1,3}},  \ldots, x^{r_{k,1}} y^{r_{k,2}} z^{r_{k,3}}\}$ such that   
  $$r_{1,i} \geq r_{2,i} \geq \cdots \geq r_{k, i}\; \; \text{and} \; \; r_{1,j} \geq r_{2,j} \geq \cdots \geq r_{k, j},$$ 
    for some $1\leq i\neq j \leq 3$. 
     Then    $\mathfrak{m}\notin \mathrm{Ass}(R/I)$, that is, $\mathrm{C}_R(I)=\emptyset$. 
\end{proposition}

\begin{proof} 
Without loss of generality, we may assume that  $r_{1,i} \geq r_{2,i} \geq \cdots \geq r_{k, i}$ for $i=1,2$. We let 
  $\mathfrak{m}\in \mathrm{Ass}(R/I)$ and seek a contradiction. This yields that  there exists some monomial $f\in R$ such that  $f\notin I$ and 
  $xf,yf,zf\in I$.  Assume $f=x^{\alpha}y^{\beta}z^{\gamma}$  with $\alpha, \beta , \gamma\geq 0$.
   This implies that  $x^{\alpha+1}y^{\beta}z^{\gamma}\in I$ and   $x^{\alpha}y^{\beta+1}z^{\gamma}\in I$. 
Accordingly,  there exist  $1\leq c \leq k$  and $1\leq d \leq k$  such that 
  \begin{equation} \label{15}
  x^{r_{c,1}} y^{r_{c,2}} z^{r_{c,3}} \mid x^{\alpha+1}y^{\beta}z^{\gamma} \;  \text {and} \;  
  x^{r_{d,1}} y^{r_{d,2}} z^{r_{d,3}} \mid  x^{\alpha} y^{\beta+1}z^{\gamma}.  
    \end{equation}
       In light   of    $x^{\alpha}y^{\beta}z^{\gamma} \notin I$,   we can conclude that 
      \begin{equation} \label{16}
       x^{r_{c,1}} y^{r_{c,2}} z^{r_{c,3}} \nmid x^{\alpha}y^{\beta}z^{\gamma} \;  \text {and} \;  
  x^{r_{d,1}} y^{r_{d,2}} z^{r_{d,3}} \nmid  x^{\alpha} y^{\beta}z^{\gamma}.
      \end{equation}
     It follows now from (\ref{15}) and (\ref{16})  that  $\alpha < r_{c,1} \leq \alpha+1$, $\beta < r_{d,2} \leq \beta+1$, 
   $r_{c, 2} \leq \beta$, and $r_{d, 1} \leq \alpha$. We thus get  $r_{d, 1} \leq \alpha < r_{c,1}$ and 
   $r_{c, 2} \leq \beta < r_{d,2}$. From  the assumption, we obtain $d<c$ and  $c<d$, which is  a contradiction. We therefore deduce  that   $\mathfrak{m}\notin \mathrm{Ass}(R/I)$, as claimed.  
\end{proof}


We now turn our attention to the existence of the maximal ideal by means of the notion of depth. 
For this purpose, we need to recall  the well-known Depth Lemma, see for instance \cite[Lemma 2.3.9]{V1}. 

\begin{lemma}\label{dep}(Depth Lemma)\index{Depth Lemma}
If $0 \rightarrow U \rightarrow M \rightarrow N \rightarrow 0$ is a short exact sequence of finitely generated graded $R$-modules, then the following statements hold:
\begin{enumerate}
\item[(1)] $\depth M \geq \min\{\depth N,\depth U\}$.
\item[(2)] $\depth U \geq \min\{\depth M,\depth N +1\}$.
\item[(3)] $\depth N \geq \min\{\depth U-1,\depth M\}$.
\end{enumerate}
\end{lemma}


In order to prove several of the results below, we will make use of the following auxiliary lemma, which we state here for convenience.

\begin{lemma}(\cite[Corollary 1.3]{asia})\label{asia}
Let $I \subset S=K[x_1,\ldots,x_n]$ be a monomial ideal and let $u$ be a monomial with $u \notin I$. Then
\[
\depth\big(S/(I : u)\big) \geq \depth(S/I).
\]
\end{lemma}


\begin{theorem}\label{t1}
Let $p\geq 0$ and $a>i_p>i_{p-1}>\cdots>i_1>i_0=0$ be some integers. Let $I_j\subset R'=K[x_1,\ldots,x_{n-1}]$, $0\leq j\leq p$, be some nonzero 
proper monomial ideals. We consider the monomial ideal 
 $$I=(x_n^a)+x_n^{i_p}I_p + \cdots + x_n^{i_1}I_1 + I_0 \subset R=K[x_1,\ldots,x_n],$$ and we assume that
$$\mathcal{G}(I)=\{x_n^a\}\cup \mathcal{G}(x_n^{i_p}I_p) \cup \cdots \cup \mathcal{G}(x_n^{i_1}I_1) \cup \mathcal{G}(I_0).$$
Let $L_j=I_0+\cdots+I_j$  for $0\leq j\leq p$.  Then the following statements hold:

\begin{enumerate}
\item[(1)] $\depth(R/I)=\depth(R'/L_p)$ or there exists $0\leq q<p$ such that
           \begin{enumerate}
           \item $\depth(R/I)=\depth(R'/L_q)$,
           \item $\depth(R/I)\leq \depth(R'/L_j)+1$  for $0\leq j\leq q-1$, where $q>0$, 
					 \item $\depth(R/I)\leq \depth(R'/L_j)$  for $q+1\leq j\leq p-1$, where $q\leq p-2$,  and
					 \item $\depth(R/I)\leq \depth(R'/L_p)-1$.
					 \end{enumerate}
\item[(2)] $\depth(R/I)=0$ if and only if there exists $0\leq q\leq p$ such that $\depth(R'/L_q)=0$.

\end{enumerate}
\end{theorem}

\begin{proof}
(1) For $1\leq j\leq p$, we consider the short exact sequences
\begin{equation}\label{seq}
0 \to \frac{R}{(I:x_n^{i_j})} \to \frac{R}{(I:x_n^{i_{j-1}})} \to \frac{R}{((I:x_n^{i_{j-1}}),x_n^{i_j-i_{j-1}})} \to 0.
\end{equation}
Note that, for all $1\leq j\leq p$, we have
 $$((I:x_n^{i_{j-1}}),x_n^{i_j-i_{j-1}})=(L_{j-1},x_n^{i_j-i_{j-1}}).$$
In addition, it is easy to see that  $(I:x_n^{i_p})=(x_n^{a-i_p},L_p)$. 
 Since $x_n^{i_j-i_{j-1}}$ is regular on $R/L_{j-1}R$, for $1\leq j\leq p$, it follows that 
\begin{align*}
\depth(R/((I:x_n^{i_{j-1}}),x_n^{i_j-i_{j-1}})) &= \depth(R/(L_{j-1},x_n^{i_j-i_{j-1}})) \\
 &= \depth(R/L_{j-1}R)-1=\depth(R'/L_{j-1}).
\end{align*}
Similarly, since $x_n^{a-i_p}$ is regular on $R/L_pR$, this implies that 
$$\depth(R/(I:x_n^{i_p}))=\depth(R/(x_n^{a-i_p},L_p)) = \depth(R'/L_p).$$
 Thanks to  $x_n^{i_p} \notin I$, one can now deduce from Lemma \ref{asia}  that 
$$\depth(R'/L_p)=\depth(R/(I:x_n^{i_p}))\geq \depth(R/I).$$
If $\depth(R/(I:x_n^{i_p}))=\depth(R/I)$, then $\depth(R/I)=\depth(R'/L_p)$ and there is nothing to prove. 
From Lemma \ref{asia}, we have 
$$\depth(R/I)\leq \depth(R/(I:x_n^{i_1}))\leq \cdots \leq \depth(R/(I:x_n^{i_p})).$$
If  $\depth(R/(I:x_n^{i_p}))>\depth(R/I)$, then  there exists $0\leq q<p$ such that
$$\depth(R/I)=\cdots=\depth(R/(I:x_n^{i_q}))<\depth(R/(I:x_n^{i_{q+1}})).$$
Taking $j=q+1$ in \eqref{seq}, we obtain the short exact sequence
$$0\to R/(I:x_n^{i_{q+1}}) \to R/(I:x_n^{i_{q}}) \to R/(L_q,x_n^{i_{q+1}-i_q}) \to 0.$$
Since $\depth(R/(I:x_n^{i_{q+1}}))>\depth(R/(I:x_n^{i_{q}}))=\depth(R/I)$, we can derive from   Lemma \ref{dep} that 
$$\depth(R/(L_q,x_n^{i_{q+1}-i_q}))=\depth(R'/L_q)=\depth(R/I).$$
Thus, we proved (a).  If $q>0$, by  applying  Lemma \ref{dep} to \eqref{seq},  we get $\depth(R/I)\leq \depth(R'/L_j)+1$ for $0\leq j\leq q-1$, 
which proves (b).  If $q<p-1$, by using  Lemma \ref{dep} to \eqref{seq},  we obtain  $\depth(R/I)\leq \depth(R'/L_j)$  for $q+1\leq j\leq p-1$, 
which proves (c). 
Finally, we have $$\depth(R'/L_p)=\depth(R/(I:x_n^{i_p}))\geq \depth(R/I)+1,$$ hence we get (d).

(2) If $\depth(R/I)=0$, then  it follows from (1) that there exists some $0\leq j\leq p$ with $\depth(R/I)=\depth(R'/L_j)=0$, as required.
    Conversely, assume that there exists some $0\leq j\leq p$ with $\depth(R'/L_j)=0$. Let $u\in R'$ be a monomial such that
		$(L_j:u)=(x_1,\ldots,x_{n-1})$. This gives  that  $$(x_1,\ldots,x_{n-1},x_n^{a-i_j}) \subset (I:x_n^{i_j}u).$$ 
 If  $j<p$, on account of  $$(I:x_n^{i_j}u) = (x_n^{a-i_j})+x_n^{i_p-i_j}(I_p:u)+\cdots+x_n^{i_{j+1}-i_{j}}(I_{j+1}:u)+(L_j:u),$$
	we  deduce that   $(x_1,\ldots,x_{n-1},x_n^{a-i_j}) \subset (I:x_n^{i_j}u)\subset (x_1,\ldots,x_{n-1},x_n^{i_{j+1}-i_j}).$ 
If $j=p$, then $(I:x_n^{a-1}u) = (x_n) + (L_p:u) = (x_1,\ldots, x_{n-1},x_n)=\mathfrak{m}$. 
	Thus,  $\sqrt{(I:x_n^{i_j}u)}=(x_1,\ldots,x_n)=\mathfrak m$. Hence, $\mathfrak m\in \Ass(R/I)$, and  so  $\depth(R/I)=0$.
		\end{proof}


\begin{remark}\label{rem1}
Let $I\subset R=K[x_1,\ldots,x_n]$ be a proper  nonzero  monomial ideal. Then $\depth(R/I)=n-1$ if and only if $I$ is principal and, otherwise, $\depth(R/I)\leq n-2$.
\end{remark}


The next corollary  establishes the third criterion of our framework.

\begin{corollary}\label{p1}
Let $p\geq 0$ and $a>i_p>i_{p-1}>\cdots>i_1>i_0=0$ be some integers. Let $I_j\subset R'=K[y,z]$, $0\leq j\leq p$, be some nonzero 
proper monomial ideals. We consider the monomial ideal $$I=(x^a)+x^{i_p}I_p + \cdots + x^{i_1}I_1  +  I_0 \subset R=K[x,y,z],$$
 and we assume that
$\mathcal{G}(I)=\{x^a\}\cup \mathcal{G}(x^{i_p}I_p) \cup \cdots \cup  \mathcal{G}(x^{i_1}I_1)    \cup       \mathcal{G}(I_0).$
Then the following statements are equivalent:
\begin{enumerate}
\item[(1)] $\depth(R/I)=1$.
\item[(2)] $I_j=(u_j)$, $0\leq j\leq p$, for some monomials $u_j\in K[y,z]$, such that $u_p \mid u_{p-1} \mid \cdots \mid u_0$,
           and all the divisions are strict.
\end{enumerate}
\end{corollary}

\begin{proof}
We let $L_j:=I_0+\cdots+I_j$ for $0\leq j\leq p$. First, note that $\depth(R/I)\in \{0,1\}$, since $I$ is not principal.
According to Theorem \ref{t1}, $\depth(R/I)=1$ if and only if $\depth(R'/L_j)>0$ for all $0\leq j\leq p$, which, according
to Remark \ref{rem1} is equivalent to $L_j$ is principal  for all $0\leq j\leq p$. We claim that this is equivalent to (2).
If $L_0$ is principal, then $I_0=L_0=(u_0)$ for some monomial $u_0\in R'$.Now, assume that $L_1$ is principal.
Since $\mathcal{G}(x^{i_1}I_1)\cup \mathcal{G}(I_0)\subset \mathcal{G}(I)$, if $v\in \mathcal{G}(I_1)$,  then $u_0\nmid v$, otherwise $x^{i_1}v$ would not be a minimal monomial
generator for $I$. On the other hand, if $v\in \mathcal{G}(I_1)$ with $v\nmid u_0$ and $u_0\nmid v$,  then $\{v,u_0\}\subset \mathcal{G}(L_1)$, a contradiction,
since $L_1$ is principal. It follows that $I_1=(u_1)$ for some monomial $u_1\in R'$ with $u_1\mid u_0$ and the division is strict. Also,
we have $L_1=(u_1)$. Now, since $L_2=I_2+L_1$ with $L_1=(u_1)$ and $L_2$ principal, using a similar argument, we deduce that
$I_2=(u_2)$ for some monomial $u_2\in R'$ with $u_2\mid u_1$ and the division is strict. Continuing in the same manner, we show that
the assertion (2) holds. 

Conversely, if (2) is true, then  the ideals $L_j=I_j$ are principal  for $0\leq j\leq p$. Hence, $\depth(R'/L_j)=1$  for $0\leq j\leq p$.
According to Theorem \ref{t1}(2), it follows that $\depth(R/I)>0$. On the other hand, since $I$ is not principal, Remark \ref{rem1} implies $\depth(R/I)\leq 1$. Therefore, $\depth(R/I)=1$, as required.
\end{proof}


To illustrate the usefulness of Corollary \ref{p1}, we provide the following example, which re-proves   Proposition 2.5  in \cite{MST}. 
Before  stating it, we require to establish the next proposition.  Generally, if $I$ is a monomial ideal, and  $f_1,\ldots,f_s \in \mathcal{G}(I)$, 
then  it is possible $f_1\cdots f_s \notin \mathcal{G}(I^s)$ for some $s>1$. For example, suppose that 
\[
I=(x^{50},x^{40}y^{10},x^{39}y^{34},x^{38}y^{35},x^{37}y^{36},
x^{36}y^{37},x^{35}y^{38},x^{34}y^{39},x^{10}y^{40},y^{50}).
\]
Then one can easily check that 
\[
I^2=(y^{100},x^{10}y^{90},x^{20}y^{80},x^{40}y^{60},
x^{50}y^{50},x^{60}y^{40},x^{80}y^{20},x^{90}y^{10},x^{100}), 
\]
while $x^{50}\cdot x^{39}y^{34}= x^{89}y^{34}\notin \mathcal{G}(I^2),$ 
$x^{40}y^{10}\cdot x^{39}y^{34}= x^{79}y^{44}\notin \mathcal{G}(I^2),$ and 
$x^{39}y^{34}\cdot x^{38}y^{35} = x^{77}y^{69}\notin \mathcal{G}(I^2).$ 


\begin{proposition}\label{min.gens.2}
Let $I\subset R=K[x_1,\ldots,x_n]$ be a monomial ideal such that all the
minimal monomial generators of $I$ have the same degree. Then, for any
integer $t\geq 1$, every product of $t$ minimal monomial generators of $I$
belongs to $\mathcal{G}(I^t)$.
\end{proposition}

 \begin{proof}
Let \(f_1,\ldots,f_t\in \mathcal{G}(I)\), and set $f:=f_1\cdots f_t.$ We want to show that \(f\in \mathcal{G}(I^t)\). 
Since \(I\) is a monomial ideal, \(I^t\) is generated by all products of \(t\) minimal monomial generators of \(I\). Hence, 
 $f\in I^t .$ It remains to prove that \(f\) is a minimal monomial generator of \(I^t\).
Suppose, for contradiction, that \(f\notin \mathcal{G}(I^t)\). Then there  exists a monomial \(g\in I^t\) such that
 $g\mid f$ and $g\neq f.$ Because \(g\in I^t\), there exist \(h_1,\ldots,h_t\in\mathcal{G}(I)\) such
that $h_1\cdots h_t\mid g.$  Thus, $h_1\cdots h_t\mid f_1\cdots f_t.$ 
By assumption, all minimal monomial generators of \(I\) have the same
degree, say \(d\). Hence, we obtain $\deg(h_1\cdots h_t)=td$ and $\deg(f_1\cdots f_t)=td.$ 
Due to  \(h_1\cdots h_t\) divides \(f_1\cdots f_t\) and both monomials have the
same degree, we must have $h_1\cdots h_t=f_1\cdots f_t=f.$ Consequently,
  $f\mid g$ and $g\mid f,$ which implies that $g=f$, contradicting the assumption that \(g\neq f\). 
Therefore, $f_1\cdots f_t\in \mathcal{G}(I^t)$, as desired.  
\end{proof}


\begin{example} \label{example-for-p1}  
\em{
Let $m\geq 2$ be an integer and consider the ideal
$$I=(x^m,xy^{m-2}z,y^{m-1}z)\subset R=K[x,y,z].$$
Our aim is to show  that $\Ass(R/I^t)=\begin{cases} \{(x,y),(x,z)\},&t\leq m-1 \\ 
\{(x,y),(x,z),(x,y,z)\},&t\geq m  \end{cases}$.

It is easy to see that $I$ has the following minimal primary decomposition
 $$I = (x, y^{m-1}) \cap (x^m, y^{m-2}) \cap (x^m, z).$$
Hence, $\Ass(R/I)=\mathrm{Min}(I)=\{(x,y),(x,z)\}.$ 

By virtue of $I = (x^m)+y^{m-2}z(x,y)$,
for any $t\geq 1$, we can write  
\begin{equation}\label{eq1}
I^t = \sum_{\ell=0}^t x^{m\ell}y^{(m-2)(t-\ell)}z^{t-\ell}(x,y)^{t-\ell}.
\end{equation}
Since all the minimal monomial generators of $I$ have degree $m$, it follows that the minimal monomial
generators of $I^t$ have degree $mt$ and, by Proposition \ref{min.gens.2}, every product of $t$  minimal monomial generators of 
$I$ is in $\mathcal{G}(I^t)$.    In what follows, we show that  $\mathcal{G}(I^t)=\{v_0,v_1,\ldots,v_{s_t}\}$, where 
$s_t=\frac{t(t+3)}{2}$.  Because 
\[
(x,y)^{t-\ell}
=
(x^{t-\ell},x^{t-\ell-1}y,\ldots,y^{t-\ell}),
\]
the $\ell$-th summand in \eqref{eq1} has the minimal monomial generators
\[
x^{m\ell+j}y^{(m-1)(t-\ell)-j}z^{t-\ell},
\qquad
0\leq j\leq t-\ell.
\]
By Proposition \ref{min.gens.2}, each of these monomials belongs to
$\mathcal{G}(I^t)$, since it is the product of $\ell$ copies of $x^m$
and $t-\ell$ minimal monomial generators of $y^{m-2}z(x,y)$.
Moreover, no two of these monomials coincide. Indeed, if
\[
x^{m\ell+j}y^{(m-1)(t-\ell)-j}z^{t-\ell}
=
x^{m\ell'+j'}y^{(m-1)(t-\ell')-j'}z^{t-\ell'},
\]
then comparing the exponents of $z$ yields $\ell=\ell'$, and
consequently comparing the exponents of $x$ gives $j=j'$. Hence all
these monomials are distinct.
Since
\[
\sum_{\ell=0}^{t}(t-\ell+1)
=
\frac{(t+1)(t+2)}{2},
\]
there are exactly $\frac{(t+1)(t+2)}{2}$ such monomials. We index them
by setting
\[
s_{\ell}
=
\ell t-\frac{\ell(\ell-3)}{2},
\]
so that $s_0=0$ and  $s_{\ell+1}-s_{\ell}=t-\ell+1.$ We need only show  that no two distinct monomials in the displayed set divide each other.
Let
\[
v_{\ell,j}
=
x^{m\ell+j}y^{(m-1)(t-\ell)-j}z^{t-\ell}.
\]
Suppose that $v_{\ell,j}\mid v_{\ell',j'}.$ Comparing the exponents of $z$ gives  $t-\ell\leq t-\ell',$ and hence
 $\ell\geq \ell'.$ Now compare the total degrees in $x$ and $y$. We have
\[
\deg_x(v_{\ell,j})+\deg_y(v_{\ell,j})
=
m\ell+j+(m-1)(t-\ell)-j
=
(m-1)t+\ell,
\]
and similarly, $\deg_x(v_{\ell',j'})+\deg_y(v_{\ell',j'})=(m-1)t+\ell'.$ 
Since $v_{\ell,j}\mid v_{\ell',j'}$, the sum of the $x$- and $y$-degrees
of $v_{\ell,j}$ cannot exceed that of $v_{\ell',j'}$. Therefore,
 $(m-1)t+\ell\leq (m-1)t+\ell',$ which implies $\ell\leq\ell'.$ 
Combining this with $\ell\geq\ell'$, we obtain $\ell=\ell'.$ 
Now both monomials have the same $z$-degree and the same value of
$\ell$. The divisibility condition becomes
\[
x^{m\ell+j}y^{(m-1)(t-\ell)-j}
\mid
x^{m\ell+j'}y^{(m-1)(t-\ell)-j'}.
\]
Hence, by comparing the exponents of $x$, we get $j\leq j',$ whereas comparing the exponents of $y$ gives
$(m-1)(t-\ell)-j \leq (m-1)(t-\ell)-j',$ and therefore $j\geq j'.$ Consequently, $j=j'.$ Therefore,
$v_{\ell,j}=v_{\ell',j'},$ which shows that no two distinct monomials in the displayed set divide
each other. Accordingly, we deduce that 
$\mathcal{G}(I^t)=\{v_0,v_1,\ldots,v_{s_t}\},$ where $s_t=\frac{t(t+3)}{2}$.  
For any $0\leq \ell\leq t$ and $0\leq j\leq t-\ell$, we can write
$$v_{\ell,j}=x^{m\ell+j}u_{s_{\ell}+j},\text{ where }u_{s_{\ell}+j}=y^{(m-1)(t-\ell)-j}z^{t-\ell}.$$
It is straightforward to check that, for  any $0\leq \ell\leq t-1$, we have  
\begin{equation}\label{eq2}
u_{s_{\ell}+t-\ell}\mid u_{s_{\ell}+t-\ell-1}\mid \cdots \mid u_{s_{\ell}},
\end{equation}
and, in particular,  the divisions are strict. Now, assume that $t\leq m-1$. Since $(m-1)(t-\ell-1)\leq (m-2)(t-\ell)$, 
for any $0\leq \ell\leq t-2$, we can conclude that 
\begin{equation}\label{eq3}
u_{s_{\ell+1}} = y^{(m-1)(t-\ell-1)}z^{t-\ell-1} \mid u_{s_{\ell}+t-\ell} = y^{(m-2)(t-\ell)}z^{t-\ell},
\end{equation}
and the division is strict. In light of  \eqref{eq2} and \eqref{eq3}, we deduce that the following divisions are strict  
$$u_{s_t-1}\mid u_{s_t-2} \mid \cdots \mid u_0.$$
Accordingly,  it follows from  Corollary \ref{p1}  that $\depth(R/I^t)=1$, that is,  $\mathfrak m\notin \Ass(R/I^t)$.
Consequently, $\Ass(R/I^t)=\Min(R/I)=\{(x,y),(x,z)\}.$

 Now, assume that  $t\geq m$. We consider the monomial $h:=x^m y^{mt-m-t} z^{t-1}$.
We recall the fact that $I^t$ is generated in degree $mt$.
Note that $\deg(h)=mt-1<mt$ and therefore $h\notin I^t$. On the other hand, we have the following observations
\begin{itemize}
\item $hx = x^{m+1} y^{mt-m-t} z^{t-1} = x^m \cdot (xy^{m-2}z)\cdot (y^{m-1}z)^{t-2}\in I^t$.
\item $hy = x^m y^{mt-m-t+1} z^{t-1} = x^m\cdot (y^{m-1}z)^{t-1} \in I^t$.
\item $hz = x^m y^{mt-m-t} z^t = (xy^{m-2}z)^m\cdot (y^{m-1}z)^{t-m}\in I^t$.
\end{itemize}
Therefore $(I^t:h)=\mathfrak m$, and thus $\mathfrak m\in\Ass(R/I^t)$, as required.
}
\end{example}


\bigskip
To complete the proof of Corollary  \ref{p2}, we need the following auxiliary proposition. In general, it should be noted that for a 
monomial ideal $I$, it  is possible  $f \in \mathcal{G}(I)$ but $f^s \notin \mathcal{G}(I^s)$ for some $s>1$. For example, assume 
\[
I=(x^5,\; y^5,\; x^4y,\; xy^4,\; x^3y^2z)\subset K[x,y,z].
\]
Then  $f=x^3y^2z \in \mathcal{G}(I)$ while  $f^3=x^9y^6z^3 \notin \mathcal{G}(I^3).$

\begin{proposition} \label{min.gens}
Let $I\subset R=K[x_1, \ldots, x_n]$ be a monomial ideal such that $|\mathcal{G}(I)|\geq 2.$ Then, for every integer $t\geq 1$,
$|\mathcal{G}(I^t)|\geq 2.$
\end{proposition}

\begin{proof}
We argue by contradiction. Suppose that for some $t\geq 1$,  $|\mathcal{G}(I^t)|=1.$ 
 Then $I^t$ is a principal monomial ideal, say $I^t=(w)$ for some monomial $w$.
  Let $\mathcal{G}(I)=\{u_1,\ldots,u_m\},$ where $m\geq 2$. Define $g=\gcd(u_1,\ldots,u_m),$ 
the greatest common divisor of the minimal monomial generators of $I$. 
Since each $u_i$ is divisible by $g$, every product of $t$ generators of $I$ is divisible by $g^t$. Hence every monomial in $I^t$ is divisible by $g^t$. In particular,  $g^t\mid w,$ because $w\in I^t$. On the other hand, for every $i=1,\ldots,m$, we have $u_i^t\in I^t=(w),$ so $w\mid u_i^t .$ 
Therefore, $w\mid \gcd(u_1^t,\ldots,u_m^t).$ Since the gcd of powers of monomials satisfies
 $\gcd(u_1^t,\ldots,u_m^t)=g^t,$ we obtain $w\mid g^t.$  Combining the two divisibilities gives
 $w=g^t.$  Therefore, $I^t=(g^t).$  Now, by virtue of  $g^t\in I^t$, there exist generators $u_{i_1},\ldots,u_{i_t}\in\mathcal{G}(I)$ 
 such that $g^t=u_{i_1}u_{i_2}\cdots u_{i_t}.$ Due to  $g$ divides every generator of $I$, we can write
 $u_{i_j}=gh_j$ for some monomial $h_j$ and for every $j=1,\ldots,t$. Hence, $g^t=(gh_1)(gh_2)\cdots(gh_t),$ which gives that 
 $g^t=g^t(h_1h_2\cdots h_t).$ Cancelling $g^t$ yields that $h_1h_2\cdots h_t=1.$ Because the $h_j$ are monomials, this implies that 
 $h_1=\cdots=h_t=1.$ Thus,  we get $u_{i_1}=\cdots=u_{i_t}=g.$ Consequently,  $g$ is one of the minimal generators of $I$. Moreover, since $g$ divides every generator $u_i$ of $I$, the minimality of the generating set forces $\mathcal{G}(I)=\{g\}.$ Hence, $|\mathcal{G}(I)|=1,$
 which contradicts the assumption that $|\mathcal{G}(I)|\geq2.$ Therefore, our assumption was false, and we conclude that
 $|\mathcal{G}(I^t)|\geq2$ for every $t\geq1$.
\end{proof}


\begin{corollary}\label{K[x,y].2}
Let $I$ be a monomial ideal in $R=K[x,y]$. Then 
\[
\mathrm{Ass}(R/I)=\mathrm{Ass}(R/I^s) \qquad \text{for all } s\geq 1.
\]
\end{corollary}

\begin{proof}
Let $\mathcal{G}(I)=\{u_1,\ldots,u_m\}$. If $m=1$, then $I$ is principal, and hence
$\mathrm{Ass}(R/I)=\mathrm{Ass}(R/I^s)$ for all $s\geq 1$. Thus, assume that $m\geq 2$. 
Now, the claim follows by combining Proposition~\ref{K[x,y].1} and Proposition~\ref{min.gens}.
\end{proof}
 

\begin{corollary}\label{p2}
With the notation from Corollary \ref{p1}, if $|\mathcal{G}(I_0)|\geq 2$, then 
$$\depth(R/I^t)=0 \text{ for all }t\geq 1.$$
\end{corollary}

\begin{proof}
Since $I=(x^a)+x^{i_p}I_p+\cdots+x^{i_1}I_1+I_0$, where $I_j\subset K[y,z]$, we get 
$$I^t = (x^{at})+x^{i_1}L_t+I_0^t,$$
where $L_t\subset R$ is a monomial ideal. Since $|\mathcal{G}(I_0)|\geq 2$,  Proposition \ref{min.gens} 
 implies  that $|\mathcal{G}(I_0^t)|\geq 2$.
 Hence, $\depth(K[y,z]/I_0^t)=0$. Let $u_t\in K[y,z]$ be a monomial such that $(I_0^t:u_t)=(y,z)$.
 It follows that 
$$(x^{at},y,z) \subset (I^t:u_t) = (x^{at})+x^{i_1}(L_t:u_t) + (y,z) \subset (x^{i_1},y,z).$$
Therefore, $\sqrt{(I^t:u_t)}=(x,y,z)$, and so $\depth(R/I^t)=0$.
\end{proof}


\begin{theorem}\label{t2}
Let $p\geq 1$ and $i_p>i_{p-1}>\cdots>i_1>i_0=0$ be some integers. Let $I_j\subset R'=K[x_1,\ldots,x_{n-1}]$, $0\leq j\leq p$, be some nonzero 
proper monomial ideals. We consider the monomial ideal 
$$I=x_n^{i_p}I_p + \cdots + x_n^{i_1}I_1 + I_0 \subset R=K[x_1,\ldots,x_n],$$ and we assume that
 $\mathcal{G}(I)=\mathcal{G}(x_n^{i_p}I_p) \cup \cdots \cup \mathcal{G}(x_n^{i_1}I_1) \cup \mathcal{G}(I_0).$ 
Let $L_j=I_0+\cdots+I_j$  for $0\leq j\leq p$.
Then the following statements hold:
\begin{enumerate}
\item[(1)] $\depth(R/I)=\depth(R'/L_p)+1$ or there exists $0\leq q<p$ such that
           \begin{enumerate}
           \item $\depth(R/I)=\depth(R'/L_q)$,
           \item $\depth(R/I)\leq \depth(R'/L_j)+1$  for $0\leq j\leq q-1$,
					 \item $\depth(R/I)\leq \depth(R'/L_j)$  for $q+1\leq j\leq p$.
					 \end{enumerate}
\item[(2)] If $\depth(R/I)=0$,  then there exists $0\leq q\leq p-1$ such that  $\depth(R'/L_q)=0$.
\item[(3)] $\depth(R/I)=0$ if and only if there exists $0\leq q\leq p-1$ and a monomial $u\in I_{q+1}\setminus I_q$ such that $\depth(R'/(L_q:u))=0$.
\end{enumerate}
\end{theorem}

\begin{proof}
(1) The proof is similar to the proof of Theorem \ref{t1}. For $1\leq j\leq p$, we consider the short exact sequences
\begin{equation}\label{seq2}
0 \to \frac{R}{(I:x_n^{i_j})} \to \frac{R}{(I:x_n^{i_{j-1}})} \to \frac{R}{((I:x_n^{i_{j-1}}),x_n^{i_j-i_{j-1}})} \to 0.
\end{equation}
As in the proof of Theorem \ref{t1}, we have 
$$\depth(R/((I:x_n^{i_{j-1}}),x_n^{i_j-i_{j-1}})) = \depth(R'/L_{j-1}) \text{ for }1\leq j\leq p.$$
On the other hand, we have 
$$\depth(R/(I:x_n^{i_{p}}))=\depth(R/L_pR)=\depth(R'/L_p)+1.$$
If $\depth(R/I)=\depth(R/(I:x_n^{i_{p}}))$, then $\depth(R/I)=\depth(R'/L_p)+1$. Otherwise, using the same argument as in the
proof of Theorem \ref{t1}, it follows that there exists $q<p$ such that $\depth(R/I)=\depth(R'/L_q)$, 
$\depth(R/I)\leq \depth(R'/L_j)+1$  for $0\leq j\leq q-1$, and $\depth(R/I)\leq \depth(R'/L_j)$  for $q+1\leq j\leq p-1$.
Also, we deduce that  $$\depth(R'/L_p)=\depth(R/(I:x_n^{i_{p}}))-1 \geq \depth(R/I).$$

(2) This can be deduced from (1). Indeed, if $\depth(R/I)=0$, then the first alternative of $(1)$ cannot occur since
 $\depth(R'/L_p)+1\geq 1.$ Hence, the second alternative must hold. Therefore, there exists some $0\leq q<p$ such that
  $\depth(R/I)=\depth(R'/L_q)=0,$ which completes the proof.

(3) Since $\depth(R/I)=0$, it follows that there exists a monomial $w\in R$ such that $(I:w)=\mathfrak m=(x_1,\ldots,x_n)$.
We can write $w=x_n^du$, where $u\in R'$. Since $(I:w)=\mathfrak m$, this implies that  there exists some $0\leq q<p$ such that
$i_{q+1}>d\geq i_q$, and thus 
$$\mathfrak m = (I:w) = (x_n^{i_p-d})(I_p:u) + \cdots + (x_n^{i_{q+1}-d})(I_{q+1}:u) + (I_q:u) + \cdots + (I_0:u).$$
Since $x_n\in (x^{i_p-d})(I_p:u) + \cdots + (x_n^{i_{q+1}-d})(I_{q+1}:u)$ and $i_p>\cdots>i_{q+1}>d$, it follows
that $x_n \in (x_n^{i_{q+1}-d})(I_{q+1}:u)$ and, therefore, $d=i_{q+1}-1$ and $u\in I_{q+1}$. 
It follows that $\mathfrak m = (x_n)+(I_q:u) + \cdots + (I_0:u)$ and thus
$$(L_q:u)=(I_q:u) + \cdots + (I_0:u)=(x_1,\ldots,x_{n-1}).$$ 
Consequently, we get $u\notin L_q$ and $\depth(R'/(L_q:u))=0$, as required. 

Conversely, since $u\in I_{q+1}$, one obtain $(I:x_n^{i_{q+1}-1}u) = (x_n) + (L_q:u)$.
It follows that $\depth(R/(I:x_n^{i_{q+1}-1}u))=\depth(R'/(L_q:u))=0$ and, by virtue of Lemma \ref{asia}, we get $\depth(R/I)=0$, as required.
\end{proof}


\begin{corollary}\label{c2}
Let $p\geq 1$ and $i_p>i_{p-1}>\cdots>i_1>i_0=0$ be some integers. Let $I_j\subset K[y,z]$, $0\leq j\leq p$, be some nonzero 
proper monomial ideals. We consider the monomial ideal $$I=x^{i_p}I_p + \cdots + x^{i_1}I_1 + I_0 \subset R=K[x,y,z],$$
 and we assume that
$\mathcal{G}(I)=\mathcal{G}(x^{i_p}I_p) \cup \cdots \cup \mathcal{G}(x^{i_1}I_1) \cup \mathcal{G}(I_0).$ 
Then the following statements are  equivalent:
\begin{enumerate}
\item[(1)] $\depth(R/I)=0$.
\item[(2)] There exists $0\leq q<p$ and a monomial $u\in I_{q+1}$, $u\notin I_j$  for $0\leq j\leq q$, 
           such that $|\mathcal{G}((I_0:u)+\cdots+(I_q:u))|\geq 2$.
\end{enumerate}
\end{corollary}

\begin{proof}
We denote $L_j=I_0+\cdots+I_j$  for $0\leq j\leq p$. According to Theorem \ref{t2}(3), $\depth(R/I)=0$ if and only if 
there exists $0\leq q\leq p-1$ and a monomial $u\in I_{q+1}\setminus I_q$ such that $\depth(R'/(L_q:u))=0$. On the other hand, since $(L_q:u)$
is a nonzero proper monomial ideal of $K[y,z]$, it follows that $\depth(R'/(L_q:u))=0$ if and only if $|\mathcal{G}((L_q:u))|\geq 2$,
which completes the proof.  
\end{proof}


We now present the fourth  criterion of our approach, which is provided by the following corollary.

\begin{corollary}\label{c3}
Let $p\geq 1$ and $i_p>i_{p-1}>\cdots>i_1>i_0=0$ be some integers.
Let $u_j\in K[y,z]$, $0\leq j\leq p-1$, be some monomials such that
\[
u_{p-1}\mid u_{p-2}\mid\cdots\mid u_1\mid u_0
\]
and all the divisions are strict. Let $I_p\subset K[y,z]$ be a proper
nonzero monomial ideal such that $u_{p-1}\notin I_p$. We consider the ideal
\[
I=x^{i_p}I_p+(x^{i_{p-1}}u_{p-1},\ldots,x^{i_1}u_1,u_0)
\subset R=K[x,y,z].
\]
We assume that \[
\mathcal{G}(I)
=
\mathcal{G}(x^{i_p}I_p)\cup\cdots\cup
\mathcal{G}(x^{i_1}I_1)\cup\mathcal{G}(I_0).
\]
Then $\depth(R/I)=1$.
\end{corollary}

\begin{proof}

For $0\leq j\leq p-1$, let $L_j=I_0+\cdots+I_j.$ Since
 $u_j\mid u_{j-1}\mid\cdots\mid u_0,$ we have
\[
I_0\subseteq I_j,\quad I_1\subseteq I_j,\quad\ldots,\quad I_{j-1}\subseteq I_j.
\]
Consequently, $L_j=I_j=(u_j).$ Thus every $L_j$ is a principal monomial ideal of
 $R'=K[y,z].$ By Remark \ref{rem1}, we obtain $\depth(R'/L_j)=1>0$ for all $0\leq j\leq p-1$. 
Now, by Theorem \ref{t2}(2), if $\depth(R/I)=0$, then there would exist
some $0\leq q<p$ such that $\depth(R'/L_q)=0.$ This contradicts the above equality. Therefore,
 $\depth(R/I)>0.$ It remains to obtain an upper bound for the depth. Since
$u_{p-1}\notin I_p$, the generators coming from $x^{i_p}I_p$ cannot all
be multiples of $u_{p-1}$. Hence $I$ has at least two minimal monomial
generators and therefore $I$ is not principal. Applying Remark \ref{rem1}
in the ring $R=K[x,y,z]$, we get $\depth(R/I)\leq 3-2=1.$ 
Combining this inequality with $\depth(R/I)>0$, we conclude that 
$\depth(R/I)=1.$ 
\end{proof}


\begin{remark}
It is natural to ask, if the converse of (2) from Theorem \ref{t2} holds, that is, if $\depth(R'/L_q)=0$  for some $q<p$, then
$\depth(R/I)=0$. The answer is no. Let $I=(xy,yz,z^2)\subset R=K[x,y,z]$. With the notations from Theorem \ref{t2} (and Corollary \ref{c2}),
we have $I=xI_1+I_0$, where $I_1=(y)$ and $I_0=(yz,z^2)$. Moreover, it is easy to see that 
$$\mathcal{G}(I)=\{xy,yz,z^2\}=\mathcal{G}(xI_1)\cup\mathcal{G}(I_0).$$
Let $L_0=I_0=(yz,z^2)$ and $L_1=I_0+I_1=(y,z^2)$. We have $\depth(R'/L_0)=0$
and $\depth(R'/L_1)=0$, where $R'=K[y,z]$.  However, $\depth(R/I)=1$. 
Here $p=1$, so the only possible value is $q=0.$ Therefore, according to Theorem \ref{t2}(3), we only need to consider
 $(L_0:u)$ for monomials $u\in I_1=(y).$ Let $u\in I_1\setminus L_0$ be a monomial. Then $u=y^d$  for some $d\geq 1$,
and therefore $\depth(R'/(L_0:u))=\depth(R'/(z))=1$. Hence, there is no monomial $u\in I_1\setminus L_0$ such that
 $\depth(R'/(L_0:u))=0$. 
\end{remark}


\begin{corollary} \label{C4}
Let $p\geq 1$ and $i_p>i_{p-1}>\cdots>i_1>i_0=0$ be some integers. Let $u_j\in K[y,z]$, $0\leq j\leq p-1$, be some monomials such that
there exists $0\leq q\leq p-1$ with $u_q\mid u_{q-1} \mid  \cdots \mid u_1 \mid u_0$, the divisions are strict, and,
if $q\leq p-2$, then
\begin{enumerate}
\item[(i)]   $u_j\nmid u_k$ for all $p-1\geq k>j\geq q$,
\item[(ii)]  $((u_{q},u_{q+1},\ldots,u_{j-1}):u_j)$ is principal  for all $q+1\leq j\leq p-1$,
\item[(iii)] $((u_{q},u_{q+1},\ldots, u_{p-1}):u)$ is principal  for any $u\in \mathcal G(I_p)$.
\end{enumerate}
Let $I_p\subset K[y,z]$ be a proper nonzero monomial ideal such that $u_q,\ldots,u_{p-1}\notin I_p$. We consider the ideal
$$I=x^{i_p}I_p+(x^{i_{p-1}}u_{p-1},\ldots,x^{i_1}u_1,u_0)\subset R=K[x,y,z].$$
We assume that \[
\mathcal{G}(I)
=
\mathcal{G}(x^{i_p}I_p)\cup\cdots\cup
\mathcal{G}(x^{i_1}I_1)\cup\mathcal{G}(I_0).
\]
Then $\depth(R/I)=1$.
\end{corollary}

\begin{proof}
For $0\leq j\leq p-1$, let $L_j=I_0+\cdots+I_j$. Since $u_q\mid u_{q-1} \mid  \cdots \mid u_1 \mid u_0$, we get
$$I_0\subset I_1 \subset \cdots \subset I_q,$$
and, consequently, $L_j=I_j$ for $0\leq j\leq q$. By Remark \ref{rem1}, we obtain $\depth(R'/L_j)=1>0$ for all $0\leq j\leq q$.

Now, assume that $q\leq p-2$. Note that, for any $q+1\leq j\leq p-1$, we have the following equalities 
$$L_j=I_0+\cdots+I_j=I_q+\cdots+I_j=(u_q,u_{q+1},\ldots,u_j).$$ 
If $L_j$ is not principal, then, according to Remark \ref{rem1}, $\depth(R'/L_j)=0$. 
Since $L_q = I_q = (u_q)$  is  principal, we can deduce that $(L_q:u)$ is principal as well. 
If $q+2\leq j\leq p-1$,   then, for any monomial $u\in I_{j}\setminus L_{j-1}$, we claim that
$(L_{j-1}:u)$ is principal. Indeed, since $u\in I_{j}$, it follows that $u=u_{j}u'$ for some monomial $u'$.
Due to  (ii), the ideal $(L_{j-1}:u_{j}) = ((u_{q},u_{q+1},\ldots,u_{j-1}):u_j)$ is principal, 
and, therefore, $(L_{j-1}:u)=((L_{j-1}:u_j):u')$ is principal.
Similarly, according to (iii), for any monomial $u\in I_p\setminus L_{p-1}$, we have  $(L_{p-1}:u)$ is principal.
To apply Corollary \ref{c2}, we need to verify that the condition
corresponding to statement {\rm (2)} of Corollary \ref{c2} does not occur.
 The proof above shows that, for $q+2\leq j\leq p-1,$ we have
\[
(L_{j-1}:u)\ \text{is principal for every }u\in I_{j}\setminus L_{j-1},
\]
and  for \(j=p\), $(L_{p-1}:u)$ is principal for every $u\in I_p\setminus L_{p-1}.$ 
Therefore, there is no integer \(q\) and no monomial \(u\) satisfying the
condition in statement {\rm (2)} of Corollary \ref{c2}. Hence, $\depth(R/I)\neq 0.$ 
Since $$\mathcal{G}(I)=\mathcal{G}(x^{i_p}I_p)\cup
\mathcal{G}(x^{i_{p-1}}I_{p-1})\cup\cdots\cup\mathcal{G}(I_0),$$
and \(I_p\neq 0\), the ideal \(I\) has at least two minimal generators.
Hence \(I\) is not principal. By Remark \ref{rem1},
 $\depth(R/I)\leq 1.$ Combining this with \(\depth(R/I)\neq0\), we conclude that
$\depth(R/I)=1.$
\end{proof}



\section{Examples and applications}\label{Section4}

 In this section, we continue to investigate the presence of the maximal ideal $(x,y,z)$ in the set of associated primes of powers of monomial ideals in $K[x,y,z]$. Through a series of lemmas and five illustrative examples, we analyze the behavior of the maximal ideal for different powers of the ideal, determining precisely when it belongs to or disappears from the sets of associated primes. Each case corresponds to a distinct configuration of the minimal primary decomposition and is supported by explicit constructions and detailed proofs. To accomplish this, we commence with the following general result.

\begin{lemma} \label{Lem.Ass}
Let $I\subset R=K[x,y,z]$ be a monomial ideal, $\mathfrak{m}=(x,y,z)$ the maximal ideal, and $s\geq 1$. Then
  $\mathrm{Ass}(R/I^s)\setminus \{\mathfrak{m}\}=\mathrm{Ass}(R/I^{s+1})\setminus \{\mathfrak{m}\}$, and hence, we always have 
  one of the following cases:
\begin{itemize}
\item[(i)]  $\mathfrak{m} \in \mathrm{Ass}(R/I^s)$ and  $\mathrm{Ass}(R/I^s)=\mathrm{Ass}(R/I^{s+1})$;
\item[(ii)]   $\mathfrak{m} \notin \mathrm{Ass}(R/I^s)$ and  $\mathrm{Ass}(R/I^s)=\mathrm{Ass}(R/I^{s+1})$;
\item[(iii)]  $\mathfrak{m} \in \mathrm{Ass}(R/I^s)$ and  $\mathrm{Ass}(R/I^{s+1})=\mathrm{Ass}(R/I^s)\setminus \{\mathfrak{m}\}$;
\item[(iv)]  $\mathfrak{m} \notin \mathrm{Ass}(R/I^s)$ and  $\mathrm{Ass}(R/I^{s+1})=\mathrm{Ass}(R/I^s) \cup \{\mathfrak{m}\}$.
 \end{itemize}
 
 In particular, we have    $\mathrm{Ass}(R/I)\setminus \{\mathfrak{m}\}=\mathrm{Ass}(R/I^{s})\setminus \{\mathfrak{m}\}$
 and  $\mathrm{Ass}(R/I^s) \subseteq \mathrm{Ass}(R/I) \cup \{\mathfrak{m}\}$. 
\end{lemma}

\begin{proof} 
Let   $I=hJ$ for some monomial $h \in R$ (possibly $h=1$) and a  monomial ideal $J$ in $R$  such that if $\mathcal{G}(J)=\{u_1, \ldots, u_r\}$, then 
  $\mathrm{gcd}(u_1, \ldots, u_r)=1$.    It follows from Theorem  \ref{5.2KHN} that 
  \begin{equation} \label{11}
\mathrm{Ass}(R/I^s)=\mathrm{Ass}(R/J^s) \cup \{(t) : t \in \mathrm{supp}(h)\}.
\end{equation}
If $J$ is $\mathfrak{p}$-primary, where $\mathfrak{p}\in \{(x), (y), (z), (x,y), (x,z), (y,z), (x,y,z)\}$, then we can deduce that 
$\mathrm{Ass}(R/J^k)=\{\mathfrak{p}\}$ for all $k\geq 1$. It therefore follows from  (\ref{11}) that  $\mathrm{Ass}(R/I^s)\setminus \{\mathfrak{m}\}=\mathrm{Ass}(R/I^{s+1})\setminus \{\mathfrak{m}\}$, and so 
 the proof is done. Hence,  we assume that $J$ is not primary.
 Since   $\mathrm{gcd}(u_1, \ldots, u_r)=1$, where   $\mathcal{G}(J)=\{u_1, \ldots, u_r\}$, we get    $\mathrm{Min}(J) \subseteq \{(x,y), (x,z), (y,z)\}$.
 On account of   $\mathrm{Min}(J) \subseteq \mathrm{Ass}(R/J^\ell) \subseteq \mathrm{Min}(J) \cup \{\mathfrak{m}\}$
  for all  $\ell\geq 1$, one can conclude that   $\mathrm{Ass}(R/J^s)\setminus \{\mathfrak{m}\}=\mathrm{Ass}(R/J^{s+1})\setminus \{\mathfrak{m}\}$.
 It follows now from (\ref{11}) that   $\mathrm{Ass}(R/I^s)\setminus \{\mathfrak{m}\}=\mathrm{Ass}(R/I^{s+1})\setminus \{\mathfrak{m}\}$. 
 
   The rest assertions can be deduced at once from the first result. 
\end{proof}


We will use the following auxiliary lemma in the sequel.

\begin{lemma}\label{lemm-1}
Let $a,b,m\geq 1$ be some integers  with $a\leq b$. We consider the ideal
$J = x^{m+b}(x,y)^a + y^{m+a}(x,y)^b \subset K[x,y].$
Then $$J^s = (x,y)^{s(m+a+b)}\quad  \text{ for all } \quad s\geq \left\lceil \frac{m+a-1}{a} \right\rceil.$$
In particular, if $a=1$,   then $J^s = (x,y)^{s(m+1+b)}$  for all $s\geq m$.
\end{lemma}

\begin{proof}
Let  $s\geq \left\lceil \frac{m+a-1}{a} \right\rceil$ be fixed.
We write $$J=(x,y)^a ( (x^{m+b}) + y^{m+a}(x,y)^{b-a}).$$ Note that
\begin{equation}\label{Js}
J^s = (x,y)^{as}(x^{m+b},x^{b-a}y^{m+a},x^{b-a-1}y^{m+a+1},\ldots, y^{m+b})^s.
\end{equation}
By our assumption on $s$, we get $as\geq m+a-1$. Let $j$ be  such that \linebreak $0\leq j\leq s(m+a+b)$
and $v=x^{s(m+a+b)-j}y^j$. We prove that $v\in J^s$.  

If $j\geq s(m+b),$ then, according to \eqref{Js}, we deduce that 
$$v= x^{as+(m+b)s-j} y^{j-(m+b)s} (y^{(m+b)})^s \in J^s.$$
Now, assume that $j\leq s(m+b)-1$. By remainder theorem, we can write
$$j = q(m+b) + r,\text{ where }0\leq r \leq m+b-1.$$
Note that $0\leq q\leq s-1$. If $m+a\leq r\leq m+b-1,$ then, according to \eqref{Js}, we obtain 
$$v =  x^{as} (x^{m+b})^{s-q-1}(y^{m+b})^q (x^{m+b-r}y^r) \in J^s.$$
If $0 \leq r\leq m+a-1,$ then, according to \eqref{Js}, we can conclude that 
$$v =  x^{as-r} y^r (x^{m+b})^{s-q}(y^{m+b})^q \in J^s.$$ 
By all the above, we get the required conclusion.
\end{proof}

\begin{remark}\label{remmy}
(1) The conclusion of Lemma \ref{lemm-1} holds, if we consider the ideal 
$$J=(x,y)^a x^{m+b} + (x,y)^b y^{m+a},$$
as an ideal in a polynomial extension of $K[x,y]$, for instance, as an ideal of $R=K[x,y,z]$.

(2) If we have $a>b$, in the statement of Lemma \ref{lemm-1}, then we can deduce a similar conclusion, that is,  
$$J^s = (x,y)^{s(m+a+b)} \quad  \text{ for all }     \quad  s\geq \left\lceil \frac{m+b-1}{b} \right\rceil.$$
In particular, for $b=1$, it holds that $J^s=(x,y)^{s(m+a+1)}$  for all $s\geq m$.
\end{remark}

Assume that \(I\subset R=K[x_1,\ldots,x_n]\) is a monomial ideal.
The \textit{contraction} of \(I\) with respect to \(x_i\), for \(1\le i\le n\), denoted by $I/x_i$,
is obtained by setting \(x_i=1\) in each minimal generator of \(I\).
From \cite[Lemma~3.18(ii)]{SN}, for all $s\geq 1$, we have $(I/x_i)^s=I^s/x_i$. On the other hand, it is easy to see that 
\[
I/{x_i}=(I:x_i^\infty),
\qquad\text{where}\qquad
(I:x_i^\infty)=\bigcup_{k\geq 1}(I:x_i^k).
\]
Therefore, this discussion leads to the  following equality 
\begin{equation} \label{4.1}
(I^s:x_i^\infty)=(I:x_i^\infty)^s.
 \end{equation}

\begin{proposition}\label{cor-magic}
Let $I\subset R=K[x,y,z]$ be a monomial ideal, and assume that 
$$J:=(I:z^{\infty})=(x,y)^a x^{m+b} + (x,y)^b y^{m+a},$$
where $b\geq a\geq 1$ and $m\geq 1$. Also, we assume that any monomial $u\in \mathcal G(I)$ has the form
$u=vz^c$, where $v\in \mathcal G(J)$ and $c\geq 0$. Then we have
\begin{enumerate}
\item[(1)] $(I^s:z^{\infty})=(x,y)^{s(m+a+b)} \text{ for all }s\geq \left\lceil \frac{m+a-1}{a} \right\rceil.$
\item[(2)] For any $s\geq \left\lceil \frac{m+a-1}{a} \right\rceil$ and $0\leq j\leq s(m+a+b)$, there exists
           $c\geq 0$ such that $u=x^j y^{s(m+a+b)-j} z^c\in \mathcal G(I^s)$. Moreover, any monomial $u\in \mathcal G(I^s)$
					 is of this form.
\item[(3)] $|\mathcal G(I^s)| = s(m+a+b)+1$.
\end{enumerate}
If $a>b$,  then we get a similar conclusion, by interchanging $a$ with $b$.
\end{proposition}

\begin{proof}
(1) It follows immediately from Lemma \ref{lemm-1} and  (\ref{4.1}).

(2) Let $s\geq \left\lceil \frac{m+a-1}{a} \right\rceil$ and $0\leq j\leq s(m+a+b)$. 
Since $$v=x^j y^{s(m+a+b)-j} \in (I^s:z^{\infty}),$$ it follows that there exists some integer $c\geq 0$
such that $u:=vz^c \in I^s$. We can assume that $c\geq 0$ is minimal with this property. 
We claim that $u\in \mathcal G(I^s)$. Indeed, if this is not the case, there exists a monomial $u' \mid u$, with $u'\neq u$, 
such that $u'\in I^s$. Since $c\geq 0$ is minimal with $u\in I^s$, it follows that $u'=x^{j'}y^{k'}z^c$, where 
either $j'<j$ and $k'\leq s(m+a+b)-j$, or $j'\leq j$ and $k'<s(m+a+b)-j$. 
It follows that $x^{j'}y^{k'}\in (I^s:z^{\infty})$  with $j'+k'<s(m+a+b)$, which contradicts 
the fact that $(I^s:z^{\infty})$ is minimally generated in degree $s(m+a+b)$.

In order to prove the last assertion, it is enough to note that $I^s$ is generated by monomials of the form
$x^{s(m+a+b)-j}y^jz^c$, where $0\leq j\leq s(m+a+b)$ and $c\geq 0$. Hence, any minimal monomial generator must be of this type.

(3) It follows from (2).
\end{proof}


In what follows, we consider five cases concerning the appearance of the
maximal ideal \((x,y,z)\) among the associated primes of powers of monomial
ideals in \(K[x,y,z]\).
 
 \medskip
\noindent\textbf{Case (I):} \\
Let $I\subset R=K[x,y,z]$ be a monomial ideal, $\mathfrak{m}=(x,y,z)$, and $I=Q \cap L$ be a minimal primary decomposition of $I$ such that $\sqrt{Q} \in \{(x,y), (x,z), (y,z)\}$ and $\sqrt{L}=\mathfrak{m}$. The following example provides a concrete instance for the case 
in which there exists some positive integer $m\geq 1$ such that  $\mathfrak{m}\in \mathrm{Ass}(R/I^s)$ for all $1\leq s \leq m$,   but  
 $\mathfrak{m}\notin \mathrm{Ass}(R/I^s)$ for all $s \geq m+1$.

 \medskip
\begin{example} \label{Ex.Ass.PandN}(\cite[Example 1]{NT1})
\em{
Consider the following monomial ideal $$I=(x^{m+3}, y^{m+3}, x^{m+2}y, xy^{m+2}, x^{m+1}y^{2}z^r),$$
 with $m, r\geq 1$ in the polynomial ring   $R=K[x,y,z]$ and  $\mathfrak{m}=(x,y,z)$.  
 It follows that 
\[
\mathrm{Ass}(R/I^s) =
    \begin{dcases}
    \{(x,y), \mathfrak{m}\} & \text{if } 1\leq s\leq m\\
\{(x,y)\}  & \text{if } s\geq m+1. \\
                \end{dcases}
\]
}
\end{example}


\medskip
\noindent\textbf{Case (II):}\\
Let $I\subset R=K[x,y,z]$ be a monomial ideal, $\mathfrak{m}=(x,y,z)$, and $I=Q_1 \cap Q_2$ be a minimal primary decomposition of $I$ such that 
$\sqrt{Q_1} , \sqrt{Q_2}\in \{(x,y), (x,z), (y,z)\}$.   The following  example serves as concrete evidence for the case in which there exists some 
positive integer $m\geq 1$ such that  $\mathfrak{m}\notin \mathrm{Ass}(R/I^s)$ for all $1\leq s \leq m$,   but  
 $\mathfrak{m}\in \mathrm{Ass}(R/I^s)$ for all $s \geq m+1$.

  \medskip
\begin{example}\label{Ex.Ass.NandP}
\em{
Let $I=(x^{2m+2}, zy^{2m+2}, zxy^{2m}, x^{2m}y)$ with $m\geq 1$  be a monomial ideal in  $R=K[x,y,z]$. We prove  that 

\[
\mathrm{Ass}(R/I^s) =
    \begin{dcases}
    \{(x,y), (x,z)\} & \text{if } 1\leq s\leq m\\
\{(x,y), (x,z),  \mathfrak{m}\}  & \text{if } s\geq m+1. \\
                \end{dcases}
\]
It is not hard to see that   $$I=(x^{2m+2}, y) \cap (x, y^{2m+2}) \cap (x^{2m}, y^{2m}) \cap (x^{2m}, z).$$
This gives that $\mathrm{Ass}(R/I)=\{(x,y), (x,z)\}$. According to   Lemma \ref{Lem.Ass}, we  can deduce   that 
 $ \{(x,y), (x,z)\} \subseteq \mathrm{Ass}(R/I^t) \subseteq \{(x,y), (x,z), \mathfrak{m}\}$ for all $t\geq 1$. 

In the first  step, fix $s\geq m+1$. In order to establish $\mathfrak{m} \in \mathrm{Ass}(R/I^s)$, set $\lambda:=x^{2m}y^{(s-1)(2m+2)}z^{s-1}$. 
 We thus get the following 
 \begin{enumerate}
  \item   Since $(zy^{2m+2})^{s-2} (zxy^{2m})(x^{2m}y) \mid x^{2m+1} y^{(s-1)(2m+2)}z^{s-1},$  we can deduce that  $x\lambda \in I^s$. 
  \item   Due to $(zy^{2m+2})^{s-1} (x^{2m}y) \mid x^{2m} y^{(s-1)(2m+2)+1}z^{s-1},$  we obtain   $y\lambda \in I^s$. 
  \item  It follows from $(zy^{2m+2})^{s-m-1} (zxy^{2m})^{m+1} \mid x^{2m} y^{(s-1)(2m+2)}z^{s}$ that  $z\lambda \in I^s$. 
  \end{enumerate}
  This implies that $(x,y,z) \subseteq (I^s:\lambda)$ for all $s\geq m+1$. We now establish $\lambda \notin I^s$. Suppose, on the contrary, that $\lambda \in I^s$. 
   Thus, there exist nonnegative integers $a,b,c,$ and $d$ with $a+b+c+d=s$  such that 
   $$(x^{2m+2})^a  (zy^{2m+2})^b  (zxy^{2m})^c  (x^{2m}y)^d \mid x^{2m} y^{(s-1)(2m+2)} z^{s-1}.$$
  In particular, we have the following conditions:
  
  \begin{itemize}
\item[(i)] $(2m+2)a+ c+ 2md\leq 2m$;
\item[(ii)] $(2m+2)b+ 2mc+ d \leq (s-1)(2m+2)$; 
\item[(iii)] $b+c \leq s-1$;
\item[(iv)] $a+b+c+d=s$, where $a,b,c,d \geq 0$. 
\end{itemize}
  
 Based on (i), we must have $d\in \{0,1\}$. We first assume that $d=1$. Then, by (i), we can deduce that $a=0$ and $c=0$. It follows from (iv) that $b=s-1$. 
 Now, the left hand side of (ii) is $2ms-2m+2s-1$, while the right hand side of (ii) is $2ms-2m+2s-2$. This leads to a contradiction. Therefore, $d\neq 1$, and so 
 $d=0$. We can conclude from (i) that $a=0$, and by (iv), we get $b+c=s$. On the other hand, by (iii), we have $b+c \leq s-1$, this gives the required contradiction. Accordingly, $\lambda \notin I^s$, and so $\mathfrak{m} \in \mathrm{Ass}(R/I^s)$ for all $s\geq m+1$.  
  
  In the  second  step,  we want to  prove   $\mathfrak{m} \notin \mathrm{Ass}(R/I^s)$ for all  $1\leq s \leq m$. 
  Let $1\leq s\leq m$. Note that $I=x^{2m}(x^2,y)+zy^{2m}(x,y^2)$. Therefore,
\begin{equation}\label{Ies}
I^s = \sum_{\ell=0}^s x^{2m(s-\ell)}y^{2m\ell}z^{\ell}(x^2,y)^{s-\ell}(x,y^2)^{\ell}.
\end{equation}
We claim that
\begin{equation}\label{eLes}
L_s:=(x^2,y)^{s-\ell}(x,y^2)^{\ell} = x^{\ell}(x^2,y)^{s-\ell}+y^{s-\ell}(x,y^2)^{\ell} = (w_{\ell,0},w_{\ell,1},\ldots,w_{\ell,s}),
\end{equation}
where $w_{\ell,j}:=\begin{cases} x^{2s-\ell-2j}y^j, & 0\leq j\leq s-\ell \\ x^{s-j}y^{\ell-s+2j},& s-\ell+1\leq j \leq s \end{cases}$.
Since $x^{\ell} \in (x,y^2)^{\ell}$, it follows that $x^{\ell}(x^2,y)^{s-\ell}\subset L_s$.
Similarly, as $y^{s-\ell}\in (x^2,y)^{s-\ell}$, this yields that   $y^{s-\ell}(x,y^2)^{\ell}\subset L_s$.
Hence, we proved $\supseteq$ in \eqref{eLes}. 

In order to prove the other inclusion, we choose $0\leq j\leq s-\ell$
and $0\leq k\leq \ell$. It is enough to show that 
$$u=x^{2(s-\ell-j)}y^{j}\cdot x^{\ell-k}y^{2k} = x^{2s-\ell-2j-k}y^{j+2k} \in x^{\ell}(x^2,y)^{s-\ell}+y^{s-\ell}(x,y^2)^{\ell}.$$
We claim that $w_{\ell,j+k}\mid u$. If $j+k\leq s-\ell$, then we get 
$$w_{\ell,j+k}=x^{2s-\ell-2j-2k}y^{j+k}\mid x^{2s-\ell-2j-k}y^{j+2k} = u,$$   as required. 
If $j+k>s-\ell$, then $w_{\ell,j+k}=x^{s-j-k}y^{\ell-s+2j+2k}$. Since $j\leq s-\ell$, it follows that
$2s-\ell-2j-k\geq s-j-k$ and, similarly, $j+2k\geq \ell-s+2j+2k$. Therefore, $w_{\ell,j+k}\mid u$, as required. 
Hence, we proved \eqref{eLes}. From \eqref{Ies} and \eqref{eLes},  it follows that
$$I^s = \sum_{\ell=0}^s x^{2m(s-\ell)}y^{2m\ell}z^{\ell} (w_{\ell,0},w_{\ell,1},\ldots,w_{\ell,s}).$$
For $0\leq \ell\leq s$ and $0\leq j\leq s$, we denote
\begin{align*}
& u_{j+\ell (s+1)}:= x^{2m(s-\ell)}y^{2m\ell}z^{\ell}\cdot w_{\ell,j}\in R\text{ and }\\
& v_{j+\ell s}:=u_{j+\ell s}/x^{\deg_x(u_{j+\ell s})}\in K[y,z].
\end{align*}
This implies that  $I^s=(u_0,u_1,\ldots,u_{s^2+2s})$. Moreover, $u_0=x^{(2m+2)s}$. 
From the definition of $w_{\ell,j}$, it is easy to see
that 
\begin{align*}
& 2s-\ell = \deg_x(w_{\ell,0})>\deg_x(w_{\ell,1})>\cdots >\deg_x(w_{\ell,s}) = 0,  \text{ and}\\
& 0 = \deg_y(w_{\ell,0})<\deg_y(w_{\ell_1})< \cdots < \deg_y(w_{\ell,s}) = \ell+s.
\end{align*}
On the other hand, since $s\leq m$, for $0\leq\ell\leq s-1$, we have 
\begin{align*}
& \deg_x(u_{s + \ell(s+1)}) = 2m(s-\ell) >  \deg_x (u_{(\ell+1)(s+1)}) = 2m(s-\ell-1)+2s-\ell-1,\\
& \deg_y(u_{s + \ell(s+1)}) = 2m\ell + \ell+s < \deg_y (u_{(\ell+1)(s+1)}) =2m(\ell+1).
\end{align*}
From all the above, it follows that
\begin{align*}
&(2m+2)s=\deg_x(u_0)>\deg_x(u_1)>\cdots>\deg_x(u_{s^2+2s})=0,  \text{ and }\\
& 0=\deg_y(u_0)<\deg_y(u_1)<\cdots<\deg_y(u_{s^2+2s}).
\end{align*}
Thus, $\mathcal G(I^s)=\{u_0,u_1,\ldots,u_{s^2+s}\}$.
Also, it is routine to check that  $$\deg_z(u_0)\leq \deg_z(u_1) \leq \cdots \leq \deg_z(u_{s^2+2s}).$$ Thus, 
$v_1\mid v_2 \mid \cdots \mid v_{s^2+2s}$ and the divisions are strict as well. We write 
$$I^s = (x^{(2m+2)s})+x^{i_1}(v_1)+ \cdots + x^{i_{s^2+2s-1}}(v_{s^2+2s-1})+(v_{s^2+2s}),$$
where $i_j=\deg_x(u_j)$. From Corollary  \ref{p1}, we deduce that  $\depth(R/I^s)=1$, and therefore $\mathfrak m\notin \Ass(R/I^s)$, 
as desired.
   }
\end{example}


 \medskip
\noindent\textbf{Case (III):} \\
Suppose that $I\subset R=K[x,y,z]$ is a monomial ideal, $\mathfrak{m}=(x,y,z)$, and $I=Q_1 \cap Q_2 \cap Q_3$
 is a minimal primary decomposition of $I$ such that 
$\sqrt{Q_1} , \sqrt{Q_2}, \sqrt{Q_3}\in \{(x,y), (x,z), (y,z)\}$.    The next  example acts as strong evidence for the case in which there exists some  positive integer  $m\geq 2$ such that  $\mathfrak{m}\notin \mathrm{Ass}(R/I^s)$ for all $1\leq s \leq m-1$,   but  
 $\mathfrak{m}\in \mathrm{Ass}(R/I^s)$ for all $s \geq m$.

    \medskip
\begin{example} \label{Ex.Ass.NandP.3}
\em{
Assume that  $I=(x^{2m+1}z, x^my^4, x^{m+1}y^2, y^{2m+1}z)$ with $m\geq 2$  is  a monomial ideal in  $R=K[x,y,z]$ and 
$\mathfrak{m}=(x,y,z)$.   We show  that 
\[
\mathrm{Ass}(R/I^s) =
    \begin{dcases}
    \{(x,y), (x,z), (y,z)\} & \text{if } 1\leq s\leq m-1\\
\{(x,y), (x,z), (y,z), \mathfrak{m}\}  & \text{if } s\geq m. \\
                 \end{dcases}
\]
One can easily see that $$I=(x^{2m+1}, y^2) \cap (y^2,z) \cap (x^m,z) \cap (x^m, y^{2m+1}) \cap (x^{m+1}, y^4).$$ 
This yields that $\mathrm{Ass}(R/I)=\{(x,y), (x,z), (y,z)\}$.  It follows from Lemma \ref{Lem.Ass} that
 $ \{(x,y), (x,z), (y,z)\} \subseteq \mathrm{Ass}(R/I^t) \subseteq \{(x,y), (x,z), (y,z), \mathfrak{m}\}$ for all $t\geq 1$. 
 
In the first step, let $s\geq m$. Our aim is to show $\mathfrak{m} \in \mathrm{Ass}(R/I^s)$. 
Put $h:=x^{m^2} y^{4m-1 + (s-m)(2m+1)} z^{s-m}$.  Hence, we have the following 
 \begin{enumerate}
  \item   On account of  $$(x^m y^4)^{m-1} (x^{m+1}y^2) (y^{2m+1}z)^{s-m} \mid x^{m^2+1} y^{4m-1 + (s-m)(2m+1)} z^{s-m},$$
    we can deduce that  $xh \in I^s$. 
  \item   Due to $(x^my^4)^m (y^{2m+1}z)^{s-m} \mid   x^{m^2} y^{4m-1 + (s-m)(2m+1)+1} z^{s-m},$  we get   $yh \in I^s$. 
  \item  According to $$(x^{m+1}y^2)^{m-1} (y^{2m+1}z)^{s-m+1} \mid x^{m^2} y^{4m-1 + (s-m)(2m+1)} z^{s-m+1},$$
   we derive that   $zh \in I^s$.  
  \end{enumerate}
  This gives  that  $(x,y,z) \subseteq (I^s:h)$ for all $s\geq m$. In what follows, we verify that $h\notin I^s$. Otherwise, we have $h\in I^s$. Hence, there exist 
  nonnegative integers $a,b,c,$ and $d$ with   $a+b+c+d=s$    such that 
  $$(x^{2m+1}z)^a  (x^my^4)^b (x^{m+1}y^2)^c  (y^{2m+1}z)^d \mid x^{m^2} y^{4m-1 + (s-m)(2m+1)} z^{s-m}.$$
   In particular, we have the following conditions:
  
  \begin{itemize}
\item[(i)] $(2m+1)a + mb + (m+1)c \leq m^2$;
\item[(ii)] $4b+2c+(2m+1)d \leq 4m-1+(s-m)(2m+1)$; 
\item[(iii)] $a+d \leq s-m$;
\item[(iv)] $a+b+c+d=s$. 
\end{itemize}
  
 It follows from (iii) and (iv) that $m\leq b+c$. If $b> m$, then $bm > m^2$, which contradicts (i). Thus, $b\leq m$. From $m\leq b+c$ and (i), we get 
  $$m^2 \geq (2m+1)a + mb + (m+1)c \geq (2m+1)a + c+ m^2.$$ This implies that $0 \geq (2m+1)a+c$. Thus, we must have $a=c=0$. 
  Since $a=0$, we can deduce from (iii) that $d\leq s-m$. It follows from $d\leq s-m$, $a=c=0$,  and (iv) that $m\leq b$. Since $b\leq m$ and $m\leq b$, we obtain $b=m$, and by (iv), one has $d=s-m$.  Thanks to $b=m$, $d= s-m$, $c=0$, and (ii), we get $0\leq -1$, which is a contradiction. Consequently, $h\notin I^s$, and therefore 
   $\mathfrak{m} \in \mathrm{Ass}(R/I^s)$ for all $s\geq m$.
   
    In the  second  step,  we prove  $\mathfrak{m} \notin \mathrm{Ass}(R/I^s)$ for all  $1\leq s \leq m-1$.
    If $m=2$, then $s=1$, and therefore there is nothing to prove. Assume that $m\geq 3$.
We denote $g_1=x^{2m+1}z$, $g_2=x^{m+1}y^2$, $g_3=x^{m}y^4$ and $g_4=y^{2m+1}z$. 
For $0\leq j\leq s-1$, we define $u_j:=g_1^{s-j}g_2^j = x^{(2m+1)s-mj}y^{2j}z^{s-j}$. 
For $0\leq \ell \leq s$ and $0\leq j \leq s-\ell$, we define  $u_{a(\ell)+j}:=g_2^{s-\ell-j}g_3^jg_4^{\ell},$
where $a(0)=s$ and $a(\ell+1)=a(\ell)+s-\ell+1$  for all $0\leq \ell\leq s-1$.
It is easy to check that $a(\ell)=(\ell+1)s + \frac{\ell(3-\ell)}{2}$ and, in particular, 
$a(s)=\frac{s(s+5)}{2}$. We claim that
$$I^s=(u_0,u_1,\ldots,u_{a(s)}).$$
Since, by definition, $u_j\in I^s$ for all $0\leq j\leq a(s)$, it is enough to show 
that $I^s\subset (u_0,u_1,\ldots,u_{a(s)})$. In order to do this, we choose an arbitrary monomial 
generator $u$ of $I^s$  and we
prove that there exists some $0\leq j\leq a(s)$ with $u_j\mid u$. 
Let $a,b,c,d\geq 0$ with $a+b+c+d=s$ and let $u:=g_1^ag_2^bg_3^cg_4^d$. 
We consider two cases: (i) $a\geq d$ and (ii) $a<d$.

(i) Assume that $a\geq d$. Let $u'=g_1^{a-d}g_2^{b+d}g_3^{c+d}$. Note that 
    \begin{align*}
    & \deg_x(g_1g_4)=2m+1 =\deg_x(g_2g_3),\\
		& \deg_y(g_1g_4)=2m+1\geq 6 =\deg_y(g_2g_3),\\
		& \deg_z(g_1g_4)=2>0=\deg_z(g_2g_3).
    \end{align*}
		Since $(g_1g_4)^{d}u' = (g_2g_3)^{d}u$, it follows that $u'\mid u$. Hence, we can assume that $d=0$, that is,
		$$u=g_1^ag_2^bg_3^c,\text{ where }a+b+c=s.$$
  	If $a=0$, then $u=g_2^{s-c}g_3^{c}=u_{a(0)+c}=u_{s+c}$, and we are done. 		
		If $c=0$, then $u=g_1^{s-b}g_2^{b}=u_b$, and we are done. 		
		Now, assume that $a,c\geq 1$. Let $u'=g_1^{a-1}g_2^{b+2}g_3^{c-1}$. Note that 
		\begin{align*}
		& \deg_x(u) - \deg_x(u') = 3m+1-2(m+1)=m-1 > 0,\\
		& \deg_y(u) - \deg_y(u') = 0\text{ and }\deg_z(u)-\deg_z(u')=1.
		\end{align*}
		Therefore, $u'\mid u$ and we can replace $u$ with $u'$. Applying the same argument, we will get $a=0$ or $c=0$, 
		and we are done.		
		
(ii) Assume that $a<d$. As in the case (i), without any loss of generality, we may assume that $a=0$, that is
     $$u=g_2^b g_3^c g_4^d = g_2^{(s-d)-c}g_3^{c}g_4^d = u_{a(d)+c},$$
		 and we are done.		         
		 
		 Note that 
		 \begin{align*}
		 & \deg_x(u_0)>\deg_x(u_1)>\cdots>\deg_x(u_s),   \text{ and }\\
		 &  \deg_y(u_0)<\deg_y(u_1)<\cdots<\deg_y(u_s).
		 \end{align*}
     Since $\deg_x(g_2)>\deg_x(g_3)$ and $\deg_y(g_2)<\deg_y(g_3)$, it follows that for any $0\leq \ell\leq s-1$, we have
		 		 \begin{align*} 
		 & \deg_x(u_{a(\ell)})>\deg_x(u_{a(\ell)+1})>\cdots>\deg_x(u_{a(\ell)+s-\ell}), \text{ and}\\
		 & \deg_y(u_{a(\ell)})<\deg_y(u_{a(\ell)+1})<\cdots<\deg_y(u_{a(\ell)+s-\ell}).
		 \end{align*}
		 Also, as $s\leq m-1$, for $0\leq \ell \leq s-1$ we obtain 
     \begin{align*}		
		  \deg_x(u_{a(\ell)+s-\ell})&=\deg_x(g_3^{s-\ell}g_4^{\ell})=(s-\ell)m > (s-\ell-1)(m+1) \\
			                          &=\deg_x(g_2^{s-\ell-1}g_4^{\ell+1}) =\deg_x(u_{a(\ell+1)}), \text{ and }\\
		  \deg_y(u_{a(\ell)+s-\ell})& =4(s-\ell) + (2m+1)\ell < 2(s-\ell-1) + (2m+1)(\ell+1) \\
		                            &  =\deg_y(u_{a(\ell+1)}).
		 \end{align*}
		 From all the above,   it follows that
		 \begin{align*}
		 & \deg_x(u_0)>\deg_x(u_1)>\cdots>\deg_x(u_{a(s)})\text{ and }\\
		 & \deg_y(u_0)<\deg_y(u_1)<\cdots<\deg_y(u_{a(s)}).
		 \end{align*}
		 Moreover, it is easy to see that 
		 $$ 0=\deg_z(u_{a(0)}) \leq \deg_z(u_{a(0)+1}) \leq \cdots \leq \deg_z(u_{a(s)}).$$
		 Now, it is clear that $\mathcal G(I^s)=\{u_0,u_1,\ldots,u_{a(s)}\}$.
		 
		 Let $v_j:=u_j / x^{\deg_x(u_j)}$ for $0\leq j\leq a(s)$. Given the expressions of $u_j$'s and the above 
		 remarks on the degrees in $x$ and $y$ of $u_j$'s, we can deduce that 
		 \begin{itemize}
		 \item $v_j = y^{2j}z^{s-j}$  for $0\leq j\leq s$.
		 \item $v_{s} \mid v_{s+1} \mid \cdots \mid v_{a(s)}$ and the divisions are strict.
		 \end{itemize}
		 Note that $I^s$ can be written as follows  
		 $$I^s = (v_0)x^{(2m+1)s} + (v_1)x^{(2m+1)s-m} + \cdots + x^{m}(v_{a(s)-1}) + (v_{a(s)}).$$
		 Let $1\leq k<j\leq s$. Then $v_j = y^{2j}z^{s-j}\nmid v_k = y^{2k}z^{s-k}$, since $j>k$. 		 
		 Also, for $0\leq j\leq s-1$, we can derive that 
		 $$(v_{j+1},v_{j+2},\ldots,v_s):v_j = (y^{2j+2}z^{s-j-1},y^{2j+2}z^{s-j-2},\ldots,v^{2s}):(y^{2j}z^{s-j}) = (y^2),$$
		 is a principal ideal. It follows now   from Corollary \ref{C4}  that $\depth(R/I^s)=1$, and thus $\mathfrak m\notin \Ass(R/I^s)$,
		 as required.
    }
\end{example}


\medskip
\noindent\textbf{Case (IV):} \\
Suppose that  $I \subset R = K[x,y,z]$ is  a monomial ideal, $\mathfrak{m} = (x,y,z)$, and
 $I = Q_1 \cap Q_2 \cap L$ is  a minimal primary decomposition of $I$ such that
$\sqrt{Q_1}, \sqrt{Q_2} \in \{(x,y), (x,z), (y,z)\}$ and $\sqrt{L}=\mathfrak{m}.$
 The following example illustrates the case in which there exists some  positive integer  $m\geq 1$ such that  $\mathfrak{m}\in \mathrm{Ass}(R/I^s)$ for all $1\leq s \leq m$,   but    $\mathfrak{m}\notin \mathrm{Ass}(R/I^s)$ for all $s \geq m+1$.

\medskip
\begin{example}  \label{Ex.Ass.PandN.2}
\em{
Let $I=(x^{m+3},\,x^{m+2}y,\,x^2y^{m+1}z^2,\,xy^{m+2}z,\,y^{m+3}z),$ where $m\geq 1$  is  a monomial ideal in  $R=K[x,y,z]$ and 
$\mathfrak{m}=(x,y,z)$.   We show  that 
\[
\mathrm{Ass}(R/I^s) =
    \begin{dcases}
    \{(x,y), (x,z), \mathfrak{m}\} & \text{if } 1\leq s\leq m\\
\{(x,y), (x,z)\}  & \text{if } s\geq m+1. \\
                 \end{dcases}
\]
It is straightforward to check that
\begin{align*}
I ={} & (x,y^{m+3})\cap (x^2,y^{m+2})\cap (x^{m+2}, y^{m+1}) \cap (x^{m+3}, y)\\
& \cap (x^{m+2},z) \cap (x^{m+2}, y^{m+2}, z^2).
\end{align*}

This implies that $\mathrm{Ass}(R/I)=\{(x,y), (x,z), (x,y,z)\}$. One can conclude from Lemma \ref{Lem.Ass}  that 
 $ \{(x,y), (x,z)\} \subseteq \mathrm{Ass}(R/I^t) \subseteq \{(x,y), (x,z), \mathfrak{m}\}$ for all $t\geq 1$. 
We first prove that $\mathfrak{m}\in \mathrm{Ass}(R/I^s)$ for all $1\leq s\leq m$. 
To do this, take $h:=x^{m+1+(s-1)(m+2)}y^{m+1}z$. We thus have the following 
\begin{itemize}
\item[$\bullet$] Since $hx = (x^{m+2}y)^s \cdot (y^{m+1-s}z)$, we have $hx \in I^s$.
\item[$\bullet$] It follows from $hy = (xy^{m+2}z) \cdot (x^{m+3})^{s-1} \cdot x^{m-s+1}$ that $hy\in I^s$.
\item[$\bullet$] Because $hz = (x^2y^{m+1}z^2) \cdot (x^{m+3})^{s-1} \cdot x^{m-s}$, we get $hz\in I^s$. 
\end{itemize}
 This implies that   $(x,y,z) \subseteq (I^s:h)$ for all $1\leq s \leq m$. We now show  that $h\notin I^s$. On the contrary, assume that  $h\in I^s$. Hence, there exist    nonnegative integers $a,b,c,d$ and $e$ with   $a+b+c+d+e=s$    such that 
  $$(x^{m+3})^a (x^{m+2}y)^b (x^2y^{m+1}z^2)^c (xy^{m+2}z)^d (y^{m+3}z)^e \mid  x^{m+1+(s-1)(m+2)}y^{m+1}z.$$
  After looking at the variables $x,y,z$, the divisibility condition implies the following conditions:
  
  \begin{itemize}
\item[(i)] $2c + d + e \le 1$;
\item[(ii)] $b + (m+1)c + (m+2)d + (m+3)e \le m+1$; 
\item[(iii)] $(m+3)a + (m+2)b +2c+d \le m+1+(s-1)(m+2)$. 
\end{itemize}
Since $b$, \(c\), \(d\), and \(e\) are nonnegative integers, it follows from (i) and (ii) that \(c = d = e = 0\).
 Due to  $a + b + c + d + e = s$, this yields that  $a = s - b$.
Substituting \(a = s - b\) into (iii), we obtain $b \ge s + 1$.  However, from the condition $a + b = s$ and the fact that $a \ge 0$, 
it must hold that $b \le s$. This yields the contradiction $s + 1 \le b \le s$. Consequently,  our initial assumption is false, and we can 
conclude that $h \notin I^s$. It follows from $(x,y,z) \subseteq (I^s : h)$ and $h \notin I^s$ that $(I^s : h) = (x,y,z)$ for all $1\leq s \leq m$, 
and so $\mathfrak{m}\in \mathrm{Ass}(R/I^s)$ for all $1\leq s\leq m$. 

Let $s\geq m+1$. In what follows, we establish $\mathfrak{m}\notin \mathrm{Ass}(R/I^{s})$. 
Since $$J=(I:z^{\infty})=(x,y) x^{m+2} + (x,y)^2 y^{m+1},$$
and $s\geq m+1$, according to Proposition \ref{cor-magic}, it follows that 
$$(I^s:z^{\infty})=(x,y)^{s(m+3)}.$$
Moreover, any $u\in \mathcal G(I^s)$ has the form $u=x^{s(m+4)-j}y^jz^c$, where $0\leq j\leq s(m+3)$ and $c\geq 0$.

We consider the monomials:
$$ u_j := \begin{cases} x^{(m+3)s-j} y^j, & 0\leq j\leq s \\ x^{(m+3)s-j}y^jz^{\left\lceil \frac{j-s}{m+2} \right\rceil}, 
& s+1\leq j\leq (m+3)s \end{cases}.$$
We demonstrate that 
$$I^s = (u_0,u_1,\ldots,u_{(m+3)s})=:L.$$
Let $0\leq j\leq s$. We can deduce that 
$$u_j = x^{(m+3)s-j}y^j = (x^{m+3})^{s-j}(x^{m+2}y)^j \in I^s.$$
Therefore, $(u_0,u_1,\ldots,u_s) \subset I^s$.

Now, we choose $j$ with $s+1\leq j\leq (m+3)s$ and we let $\ell=\left\lceil \frac{j-s}{m+2} \right\rceil$.
Note that $1\leq \ell\leq s$.
We have  $(m+2)(\ell-1)+1+s \leq  j \leq  (m+2)\ell +s,$ which is equivalent to $$ m+1 \geq  (m+2)\ell+s-j  \geq 0. $$
We denote $c:=(m+2)\ell+s-j$. 
If $c\leq \ell,$ then
$$u_j = x^{(m+3)s-j}y^jz^{\ell} = (x^{m+2}y)^{s-\ell}(xy^{m+2}z)^c(y^{m+3}z)^{\ell-c} \in I^s.$$
Similarly, if $c>\ell,$ then
$$u_j = x^{(m+3)s-j}y^jz^{\ell} = (x^{m+3})^{c-\ell}(x^{m+2}y)^{s-c}(xy^{m+2}z)^{\ell} \in I^s.$$
Hence, we proved that $L\subset I^s$.

To show the reverse inclusion, choose an arbitrary generator of $I^{s}$, say the monomial
 $$u:=(x^{m+3})^a (x^{m+2}y)^b (x^2y^{m+1}z^2)^c (xy^{m+2}z)^d (y^{m+3}z)^e,$$ with
$a,b,c,d$ and $e$ are nonnegative and $a+b+c+d+e=s$.
Note that $$u=x^{a(m+3)+b(m+2)+2c+d}y^{b+(m+1)c+(m+2)d+(m+3)e}z^{2c+d+e}.$$
We put  $\lambda=b+(m+1)c+(m+2)d+(m+3)e$. Since $a+b+c+d+e=s$, it  is straightforward to check that 
$\deg_y(u) =\lambda$ and 
$$\deg_x(u) = a(m+3)+b(m+2)+2c+d = s(m+3)-\lambda.$$
To complete our proof,  it is enough to show that $u_\lambda\mid u$. To do this, we need only prove that 
$$\ell = \left\lceil \frac{\lambda-s}{m+2} \right\rceil \leq \deg_z(u)=2c+d+e.$$
Indeed, since $a,b,c,d,e$ are non negative and $a+b+c+d+e=s$, we have 
 $b+e\leq s-a-c-d\leq s-c.$ This yields that 
\[
b-c+e-s\leq (s-c)-c-s=-2c\leq 0.
\]
Therefore, $\left\lceil \frac{b-c+e-s}{m+2} \right\rceil\leq 0.$ Hence, we obtain 
\begin{align*}
\left\lceil \frac{\lambda-s}{m+2} \right\rceil &= \left\lceil \frac{b + (m+1)c + (m+2)d + (m+3)e - s}{m+2} \right\rceil  \\
&= c+d+e + \left\lceil \frac{b - c + e -s}{m+2} \right\rceil \leq c+d+e \leq 2c+d+e. 
\end{align*}
 Consequently, we can deduce that  $I^s=(u_0,u_1,\ldots,u_{s(m+3)})$.  On the other hand, it is not hard to show that 
 for $j\neq k$, we have $u_j\nmid u_k$ and $u_k\nmid u_j$. This gives that  $\mathcal{G}(I^{s})=\{u_0,u_1,\ldots,u_{s(m+3)}\}$. 
 We can write 
$$
I^{s}
= (x^{s(m+3)}) 
+ x^{s(m+3)-1}(v_1)
+ x^{s(m+3)-2}(v_2) + \cdots + (v_{s(m+3)}),$$
where $v_j=u_j / x^{s(m+3)-j} = \begin{cases} y^j, & 0\leq j\leq s  \\ y^jz^{\left\lceil \frac{j-s}{m+2} \right\rceil}, & s+1
\leq j\leq s(m+3)
\end{cases}$. 

One can easily check that  $v_1\mid v_2\mid \cdots \mid v_{s(m+3)}$   and that all the divisions are strict.
It follows now from  Corollary \ref{p1}  that $\depth(R/I^{s})=1$. Thus, $\mathfrak m=(x,y,z)\notin \Ass(R/I^{s})$. 
Hence, $\Ass(R/I^s)=\Min(I^s)=\{(x,y),(x,z)\}$ for all $s\geq m+1$, as required.
}
\end{example}


 \medskip
\noindent\textbf{Case (V):} \\
Suppose that  $I\subset R=K[x,y,z]$ is a monomial ideal, $\mathfrak{m}=(x,y,z)$, and $I=Q_1 \cap Q_2 \cap Q_3 \cap L$ is  a minimal
 primary decomposition of $I$ such that $\sqrt{Q_1} , \sqrt{Q_2}, \sqrt{Q_3}\in \{(x,y), (x,z), (y,z)\}$ and $\sqrt{L}=\mathfrak{m}$. 
 The next example concerns the case  in which there exists some positive integer $m\geq 1$ such that 
  $\mathfrak{m}\in \mathrm{Ass}(R/I^s)$ for all $1\leq s \leq m$,   but  
 $\mathfrak{m}\notin \mathrm{Ass}(R/I^s)$ for all $s \geq m+1$.

\medskip
\begin{example}  \label{Ex.Ass.PandN.3}
\em{
Consider the following monomial ideal 
$$I=\left(x^{m+4}z,x^{m+3}y,x^{m+2}y^2,x^{m+1}y^3z^2,xy^{m+3}z,y^{m+4}z\right) \subset R=K[x,y,z],$$ 
where $m\geq 1$   and $\mathfrak{m}=(x,y,z)$.  Our aim is to show that 
\[
\mathrm{Ass}(R/I^s) =
    \begin{dcases}
    \{(x,y), (x,z), (y,z), \mathfrak{m}\} & \text{if } 1\leq s\leq m\\
\{(x,y), (x,z), (y,z)\}  & \text{if } s\geq m+1. \\
                 \end{dcases}
\]
One can easily check that 
\begin{align*}
I ={} & (x^{m+2}, z) \cap (x^{m+4}, y) \cap (x, y^{m+4}) \cap (x^{m+1}, y^{m+3}) \cap (x^{m+3}, y^2) \\
& \cap (x^{m+2}, y^{m+3}, z^2) \cap (z, y)  \cap (x^{m+2}, y^3)
\end{align*}
We thus get  $\mathrm{Ass}(R/I)=\{(x,y), (x,z), (y,z), (x,y,z)\}$. It follows from  Lemma \ref{Lem.Ass}  that 
 $ \{(x,y), (x,z), (y,z)\} \subseteq \mathrm{Ass}(R/I^t) \subseteq \{(x,y), (x,z), (y,z), \mathfrak{m}\}$ for all $t\geq 1$. 
  We first demonstrate that  $\mathfrak{m}\in \mathrm{Ass}(R/I^s)$ for all $1\leq s\leq m$. 
 Put  $h:=x^{m+1} y^{s(m+3)-1}z^s$. Hence, we get  the following 
\begin{itemize}
\item[$\bullet$] Since $xh = {(x^{m+2} y^2)} \cdot {(y^{m+4} z)^{s-1}} \cdot y^{m-s+1}z$, we have $xh \in I^s$.
\item[$\bullet$] It follows from $yh = (xy^{m+3}z)^s \cdot x^{m+1-s}$  that $yh\in I^s$.
\item[$\bullet$] Because $zh = {(x^{m+1} y^3 z^2)} \cdot {(y^{m+4} z)^{s-1}} \cdot y^{m-s},$ we get $zh\in I^s$. 
\end{itemize}
 Therefore, we deduce that    $(x,y,z) \subseteq (I^s:h)$ for all $1\leq s \leq m$. Suppose, on the contrary, that
 $h\in I^s$. Hence, there exist nonnegative integers $a,b,c,d,e$ and $f$ with $a + b + c + d + e+f = s$  such that
$$(x^{m+4}z)^a (x^{m+3}y)^b (x^{m+2}y^2)^c (x^{m+1}y^3z^2)^d (xy^{m+3}z)^e (y^{m+4}z)^f \mid x^{m+1} y^{s(m+3)-1}z^s.$$
The divisibility condition gives the following conditions:

\begin{itemize}
\item[(i)]  $(m+4)a + (m+3)b + (m+2)c + (m+1)d + e \le m+1$;
\item[(ii)]  $b + 2c + 3d + (m+3)e + (m+4)f \le s(m+3) - 1$;
\item[(iii)]   $a + 2d + e + f \le s$.
\end{itemize}
From  (i), we must have $a=b=c=0$ and $d\in \{0,1\}$. Hence, one may consider the following cases:

\medskip
\textbf{Case 1.}  $d = 0$. Then (i), (ii), and (iii) imply  $e \le m+1$, $e + f \le s$ and 
$$(m+3)e + (m+4)f \le s(m+3) - 1.$$ 
Assuming that $e+f=s$,  we can deduce the following inequality
  $$(m+3)e + (m+4)f \ge (m+3)(e+f)=s(m+3),$$
  which is a  contradiction. Therefore, we must have $e + f \le s - 1$. This yields that 
  $a + b + c + d + e + f  \le s - 1$, which contradicts $a+b+c+d+e+f=s$. 
  
  \medskip
  \textbf{Case 2.}  $d = 1$. Then (i) implies that $e=0$. It follows from (iii) that $f\leq s-2$. Hence, we obtain $
  a + b + c + d + e + f \le 1 + (s - 2) = s - 1,$ which  contradicts our assumption $a + b + c + d + e+f = s$.   
  
Accordingly, we must have  $h \notin I^s$. Since  $(x,y,z) \subseteq (I^s : h)$ and $h \notin I^s$, this leads to 
 $(I^s : h) = (x,y,z)$ for all $1 \le s \le m$, and so $\mathfrak{m} \in \operatorname{Ass}(R/I^s)$ for all $1 \le s \le m$. 
 
 Let $s\geq m+1$. In what follows, we establish $\mathfrak{m}\notin \mathrm{Ass}(R/I^{s})$. 
Since $J=(I:z^{\infty})=x^{m+1}(x,y)^3+(x,y)y^{m+3}$ and $s\geq m+1$, according to Proposition \ref{cor-magic}, it follows that 
$(I^s:z^{\infty})=(x,y)^{s(m+4)}.$
Moreover, any $u\in \mathcal G(I^s)$ has the form $u=x^{s(m+3)-j}y^jz^c$, where $0\leq j\leq s(m+3)$ and $c\geq 0$.

To simplify the notation, 
 put  $g_1:=x^{m+4}z$, $g_2:=x^{m+3}y$, $g_3:=x^{m+2}y^2$,
$g_4:=x^{m+1}y^3z^2$, $g_5:=xy^{m+3}z$ and $g_6:=y^{m+4}z$.  We want to verify that 
$I^{s}=(u_0,u_1,\ldots,u_{s(m+4)}),$ where 
\begin{equation*}
u_j=\begin{cases} x^{s(m+4)-j}y^j z^{s-j},\;0\leq j\leq s \\
                  x^{s(m+4)-j}y^j, \;s+1 \leq j\leq 2s \\ 
									x^{s(m+4)-j}y^j z^{\left\lceil \frac{j-2s}{m+2}\right\rceil}, \;2s+1 \leq j\leq s(m+4). \end{cases}
\end{equation*}
We first show that $(u_0,u_1,\ldots,u_{s(m+4)}) \subseteq I^{s}$.  We have
\begin{align*}
 u_j = x^{s(m+4)-j}y^jz^{s-j} &= (x^{m+4}z)^{s-j}(x^{m+3}y)^j =g_1^{s-j}g_2^j \in I^{s},\text{ for }0\leq j\leq s, 
\end{align*} and
\begin{align*}
 u_{j+s} = x^{s(m+3)-j}y^{j+s}& = (x^{m+3}y)^{s-j}(x^{m+2}y^2)^j=g_2^{s-j}g_3^{j} \in 
I^{s},\text{ for }1\leq j\leq s.
\end{align*}
Note that, the set of indices $\Lambda:=\{2s+1,2s+2,\ldots, s(m+4)\}$ can be written as
the union of disjoint sets as follows
\begin{align*}
 \Lambda	= & \bigcup_{1\leq\ell\leq s}\{(\ell-1)(m+2)+1+2s,(\ell-1)(m+2)+2+2s,\ldots,\ell(m+2)+2s\}\\
				  = &\bigcup_{\substack{1\leq\ell\leq m+1 \\ 0\leq j\leq (m+1)-\ell}}\{(\ell-1)(m+2)+1+2s+j\} \cup
				     \bigcup_{\substack{1\leq\ell\leq m+1 \\ 1\leq j\leq \ell}}\{\ell(m+1)+2s+j\} \\
			      & \cup \bigcup_{\substack{m+2\leq\ell\leq s \\ 0\leq j\leq m+1}}\{(\ell-1)(m+2)+1+2s+j\}.
\end{align*}
Let $1\leq \ell\leq m+1$ and $0\leq j\leq m+1-\ell$. Then we have 

$$\left\lceil \frac{(\ell-1)(m+2)+1+2s+j-2s}{m+2}\right\rceil = \ell-1 + 
\left\lceil \frac{j+1}{m+2}\right\rceil = \ell.$$
Therefore, we can conclude the following 
\begin{align*}
u_{(\ell-1)(m+2)+ 1+ 2s + j} &= x^{(s-\ell+1)(m+2)-j-1}y^{(\ell-1)(m+2)+ 1+ 2s + j}z^{\ell} \\
                             & = (x^{m+3}y)^{m+1-\ell-j}(x^{m+2}y^2)^{j+s-m-1}(xy^{m+3}z)^{\ell}\\
                             &= g_2^{m+1-\ell-j}g_3^{j+s-m-1}g_5^{\ell}\in I^{s}.
\end{align*}
Also, for $1\leq \ell\leq m+1$ and $1\leq j\leq \ell$, we obtain 
\[
\begin{aligned}
\left\lceil \frac{(\ell(m+1)+2s+j-2s}{m+2}\right\rceil
&=\left\lceil \frac{(m+2)\ell-(\ell-j)}{m+2}\right\rceil 
=\ell-\left\lfloor\frac{\ell-j}{m+2}\right\rfloor=\ell,
\end{aligned}
\]
and we therefore get 
$$u_{\ell(m+1)+2s+j} = x^{(s-\ell)(m+2)+\ell-j}y^{\ell(m+1)+2s+j}z^{\ell} = 
g_3^{s-\ell}g_5^{\ell-j}g_6^j\in I^{s}.$$
Also, for $m+2\leq\ell\leq s$ and $0\leq j\leq m+1$, we obtain
\[
\begin{aligned}
\left\lceil \frac{(\ell-1)(m+2)+1+2s+j-2s}{m+2}\right\rceil
&= \ell-1 + \left\lceil \frac{j+1}{m+2}\right\rceil 
=\ell-1+1=\ell,
\end{aligned}
\]
and we therefore get \[
\begin{aligned}
u_{(\ell-1)(m+2)+1+2s+j} &= x^{(s-\ell+1)(m+2)-j-1}y^{(\ell-1)(m+2)+1+2s+j}z^{\ell}\\
& =  g_3^{s-\ell} g_5^{m+1-j} g_6^{\ell+j-m-1} \in I^s.
\end{aligned}
\]
Accordingly, we can deduce  $(u_0,u_1,\ldots,u_{s(m+4)}) \subseteq I^{s}$. 

In order to prove the other inclusion, 
we choose  arbitrary nonnegative integers $a,b,c,d,e$ and $f$ with $a+b+c+d+e+f=s$ and we  need to show that 
 $$u:=g_1^a g_2^b g_3^c g_4^d g_5^e g_6^f\in (u_0,\ldots,u_{s(m+4)}).$$
 It is easy to see that 
\begin{align*}
u&=(x^{m+4}z)^a(x^{m+3}y)^b(x^{m+2}y^2)^c(x^{m+1}y^3z^2)^d(xy^{m+3}z)^e(y^{m+4}z)^f\\
 &=x^{(m+4)a+(m+3)b+(m+2)c+(m+1)d+e}y^{b+2c+3d+(m+3)e+(m+4)f}z^{a+2d+e+f}.
\end{align*}
Let $j:=b+2c+3d+(m+3)e+(m+4)f$. We have
\[
j=b+2c+3d+(m+3)e+(m+4)f
\leq (m+4)(b+c+d+e+f).
\]
But $b+c+d+e+f=s-a\leq s,$ and so $j\leq s(m+4).$
Hence, we can write  $u=x^{s(m+4)-j}y^jz^{a+2d+e+f}$. 
 Therefore, in order to prove that $u\in (u_0,\ldots,u_{s(m+4)})$,
it is enough to show that $$\deg_z(u) = a+2d+e+f \geq \deg_z(u_j).$$
We consider the following  cases:
\begin{itemize}
\item If $0\leq j\leq s$, then, since $j=b+2c+3d+(m+3)e+(m+4)f$, it follows that 
      \begin{align*}
      \deg_z(u_j)&=s-j=(a+b+c+d+e+f)-(b+2c+3d+(m+3)e+(m+4)f)\\
      &=a-c-2d-(m+2)e-(m+3)f \leq a \leq a+2d+e+f = \deg_z(u),
      \end{align*}
        as required.
\item If $s+1\leq j\leq 2s$, then $\deg_z(u_j)=0$ and there is nothing to prove.
\item Assume that $j\geq 2s+1$ and let $\ell=\left\lceil \frac{j-2s}{m+2}\right\rceil$. This implies that 
      $$ 2s+(m+2)\ell-m-1 \leq j \leq 2s + (m+2)\ell.$$
      Consequently, we can derive that 	
	    $$\deg_x(u_j)=s(m+4)-j \leq s(m+4) - 2s -(m+2)\ell + m+1 = (s-\ell)(m+2)+m+1.$$
				
			Due to  $\deg_x(u)=\deg_x(u_j)$ and  $\deg_x(g_2^b g_3^c) \leq \deg_x(u),$  we obtain
			$$\deg_x(g_2^bg_3^c)=(m+3)b+(m+2)c \leq \deg_x(u_j) \leq (m+2)(s+1-\ell)-1.$$
		  This yields  that $b+c\leq s-\ell$. We thus get
		  \begin{align*}
			\deg_z(u)&=a+2d+e+f = s + (a+2d+e+f) - (a+b+c+d+e+f)  \\
		           &=s+d-b-c \geq s-(b+c) \\
		          & \geq s - (s-\ell)=\ell=\deg_z(u_j),
			\end{align*}
			as required.						
\end{itemize}
Thus, we proved  that   $ I^{s}  \subseteq    (u_0,u_1,\ldots,u_{s(m+4)})$,  which completes the proof. 

For $0\leq j\leq s(m+4)$,  we consider the monomials $v_j\in K[y,z]$ as follows
$$v_j:=u_j/x^{s(m+4)-j}=\begin{cases} y^j z^{s-j},\;0\leq j\leq s \\
 y^j, \;s+1 \leq j\leq 2s \\ y^jz^{\left\lceil \frac{j-2s}{m+2}\right\rceil} , \;2s+1 \leq j\leq s(m+4) \end{cases}.$$
Observe  that $v_{s}\mid v_{s+1}\mid \cdots \mid v_{s(m+4)}$, and the divisions are strict. 
  Furthermore, it is straightforward to verify that 
\[
\begin{aligned}
I^{s}={}&x^{s(m+4)}(v_0)+x^{s(m+4)-1}(v_1)+\cdots+x(v_{s(m+4)-1})+(v_{s(m+4)}).
\end{aligned}
\]
Moreover, for every $1\leq j\leq s-1$, we have the following 
\[
\begin{aligned}
v_{j-1}=y^{j-1}z^{s+1-j}
&\notin (v_j,v_{j+1},\ldots,v_{s(m+4)})\\
&=(v_j,v_{j+1},\ldots,v_{m+1})\\
&=(y^jz^{s-j},\,y^{j+1}z^{s-1-j},\,\ldots,\,y^{s-1}z,\,y^{s}).
\end{aligned}
\]
Since $v_{j-1+k} \mid y^{k}v_{j-1}$ for $1\leq k\leq s+1-j$, and $zv_{j-1}\notin (v_j,v_{j+1},\ldots,v_{m+1})$, it
follows that  $((v_j,v_{j+1},\ldots,v_{s}):v_{j-1})=(y)$  is principal. 
 It follows from  Corollary \ref{C4} that  $\depth(R/I^{s})=1$, and thus $\mathfrak m\notin \Ass(R/I^{s})$, as required.
 }
\end{example}


\section{Fluctuations in  associated primes of powers of monomial ideals}\label{Section5}

Let $I\subset R=K[x_1, \ldots, x_n]$ be a monomial ideal.  Recall that  we say  $I$ has the phenomenon of fluctuation in associated primes of 
powers if there exist  a monomial prime ideal $\mathfrak{p}$ in $R$  and  positive integers $a<b<c$ satisfying  at least one of the following cases:
\begin{itemize}
\item[(i)] $\mathfrak{p} \in \mathrm{Ass}(R/I^a)$, $\mathfrak{p} \notin \mathrm{Ass}(R/I^b)$, and $\mathfrak{p} \in \mathrm{Ass}(R/I^c)$. 
\item[(ii)] $\mathfrak{p} \notin \mathrm{Ass}(R/I^a)$, $\mathfrak{p} \in \mathrm{Ass}(R/I^b)$, and $\mathfrak{p} \notin \mathrm{Ass}(R/I^c)$.
\end{itemize}

Since every monomial ideal in $R=K[x]$ is principal, it follows that there is no fluctuation in the associated primes of its powers. In addition, according to Corollary \ref{K[x,y].2}, we can deduce that there is no fluctuation in the associated primes of  powers of monomial ideals in $R=K[x,y]$. 

  We now study fluctuations in the associated primes of powers of monomial ideals in $R=K[x,y,z]$. We begin by presenting two counterexamples demonstrating that the associated primes of powers of monomial ideals can indeed fluctuate. In particular, these counterexamples will be used to establish three applications in the sequel.

 \begin{example} \label{Counterexample.1}
{\em 
(1) Assume the following monomial ideal in $R = K[x,y,z]$
$$I=(x^8,\,
x^7y,\,
x^6y^2z^5,\,
xy^7z^{10},\,
y^8z^{12}).$$
Put $\mathfrak{m}:=(x,y,z)$.   Using  \textit{Macaulay2} \cite{GS}, we obtain the following  
\[
\begin{aligned}
\operatorname{Ass}(R/I)
&= \{(x,y),(x,z)\}= \mathrm{Min}(I),\\
\operatorname{Ass}(R/I^2)
&= \{(x,y),(x,z)\}= \mathrm{Min}(I),\\
\operatorname{Ass}(R/I^3)
&= \{(x,y),(x,z),(x,y,z)\}=\mathrm{Min}(I)\cup \{\mathfrak{m}\},\\
\operatorname{Ass}(R/I^4)
&= \{(x,y),(x,z),(x,y,z)\}=\mathrm{Min}(I)\cup \{\mathfrak{m}\},\\
\operatorname{Ass}(R/I^5)
&= \{(x,y),(x,z),(x,y,z)\}=\mathrm{Min}(I)\cup \{\mathfrak{m}\},\\
\operatorname{Ass}(R/I^6)
&= \{(x,y),(x,z)\}= \mathrm{Min}(I),\\
\operatorname{Ass}(R/I^7)
&= \{(x,y),(x,z)\}= \mathrm{Min}(I).
\end{aligned}
\]
Thus, there exists the phenomenon of fluctuation in associated primes of powers of $I$. 
 
  (2) Let $I = (x^5, x^4yz, y^3z^5, xy^2z^4, x^3y^4z^3)$ be a monomial ideal in $R=K[x,y,z]$. 
  Using  \textit{Macaulay2} \cite{GS}, we get
 \[
\begin{aligned}
\operatorname{Ass}(R/I)&=\{(x,y),(x,z), (x,y,z)\}=\mathrm{Min}(I)\cup \{\mathfrak{m}\},\\
\operatorname{Ass}(R/I^2)&=\{(x,y),(x,z)\}=\mathrm{Min}(I),\\
\operatorname{Ass}(R/I^3)&=\{(x,y),(x,z),  (x,y,z)\}=\mathrm{Min}(I)\cup \{\mathfrak{m}\}.
\end{aligned}
\]
 Accordingly, one can see that  there exists the phenomenon of fluctuation in associated primes of powers of $I$. 
 }
 \end{example}


 At this point, one may ask the following natural question: can one find infinitely many monomial ideals in
 $R=K[x_1,\ldots,x_n],$ where $n\geq 3$, for which the associated primes of their powers exhibit a fluctuation phenomenon?
The answer is affirmative, as demonstrated by Theorem \ref{Th.Fluctuation}. To show Theorem \ref{Th.Fluctuation}, we need to utilize 
 Lemma \ref{Lem.Fluctuation}.  
To establish Lemma \ref{Lem.Fluctuation}, we first recall the notions of expansion and weighting operations. We begin with the expansion construction for monomial ideals introduced in \cite{BH}. Let $R=K[x_1,\ldots,x_n]$ be a polynomial ring over a field $K$. Fix an ordered $n$-tuple of positive integers $(i_1,\ldots,i_n),$ and consider the polynomial ring $R^{(i_1,\ldots,i_n)}$ over $K$ whose variables are
\[
x_{11},\ldots,x_{1i_1},\,
x_{21},\ldots,x_{2i_2},\,
\ldots,\,
x_{n1},\ldots,x_{ni_n}.
\]
For each $j\in\{1,\ldots,n\}$, let $\mathfrak{p}_j=(x_{j1},x_{j2},\ldots,x_{ji_j}) \subseteq R^{(i_1,\ldots,i_n)}$ 
denote the corresponding monomial prime ideal.  Now let $I\subset R$ be a monomial ideal with a set of monomial generators
 $\{\mathbf{x}^{\mathbf a_1},\ldots,\mathbf{x}^{\mathbf a_m}\},$ where $\mathbf{x}^{\mathbf a_i}
=x_1^{a_i(1)}\cdots x_n^{a_i(n)}$ and $\mathbf a_i=(a_i(1),\ldots,a_i(n)).$ 
For $i=1,\ldots,m$ and $j=1,\ldots,n$, the notation $a_i(j)$ refers to the $j$th entry of the vector $\mathbf a_i$.

The \emph{expansion of $I$ with respect to $(i_1,\ldots,i_n)$}, denoted by $I^{(i_1,\ldots,i_n)}$, is defined as the monomial ideal in $R^{(i_1,\ldots,i_n)}$ given by
\[
I^{(i_1,\ldots,i_n)}
=
\sum_{i=1}^{m}
\mathfrak p_1^{a_i(1)}
\cdots
\mathfrak p_n^{a_i(n)}.
\]
For simplicity, throughout the discussion we use $R^*$ and $I^*$ in place of
$R^{(i_1,\ldots,i_n)}$ and $I^{(i_1,\ldots,i_n)}$, respectively.

As a concrete example, consider the polynomial ring $R=K[x_1,x_2,x_3]$ and let the ordered $3$-tuple be $(2,3,1)$. 
 The corresponding monomial prime ideals in $R^*=K[x_{11},x_{12},x_{21},x_{22},x_{23},x_{31}]$ are
$\mathfrak p_1=(x_{11},x_{12}),$ $\mathfrak p_2=(x_{21},x_{22},x_{23})$ and $\mathfrak p_3=(x_{31}).$ 
Now take the monomial ideal $I=(x_1^2x_2,\;x_2^2x_3,\;x_1x_3^2).$ According to the definition of expansion, we obtain
\[
I^*
=
\mathfrak p_1^2\mathfrak p_2
+
\mathfrak p_2^2\mathfrak p_3
+
\mathfrak p_1\mathfrak p_3^2.
\]

Thus, the expansion of $I$ is the following monomial ideal:
\begin{align*}
I^*=(&
x_{11}^2x_{21},x_{11}^2x_{22},x_{11}^2x_{23},
x_{11}x_{12}x_{21},x_{11}x_{12}x_{22},x_{11}x_{12}x_{23},
x_{12}^2x_{21},\\
& x_{12}^2x_{22},x_{12}^2x_{23},x_{21}^2x_{31},x_{21}x_{22}x_{31},
x_{21}x_{23}x_{31},x_{22}^2x_{31},
x_{22}x_{23}x_{31}, \\ 
& x_{23}^2x_{31},  x_{11}x_{31}^2,x_{12}x_{31}^2).
\end{align*}


We next introduce the weighting operation for monomial ideals. We begin by recalling the definition of a weight and then describe how it induces a weighted ideal.

\begin{definition}
{\em 
Let $R=K[x_1,\ldots,x_n]$ be a polynomial ring over a field $K$. Then a \textit{weight} on $R$ is a map 
$W:\{x_1,\ldots,x_n\}\longrightarrow\mathbb{N}.$ For each $1\leq i \leq n$, we write $w_i=W(x_i)$ and refer to $w_i$ as the weight 
assigned to the variable $x_i$. Let $I\subseteq R$ be a monomial ideal and let $W$ be a weight on $R$. The \textit{weighted ideal} 
associated with $I$ and $W$, denoted by $I_W$, is defined as $I_W=(h(u):u\in\mathcal{G}(I)),$ where $\mathcal{G}(I)$ denotes 
the minimal monomial generating set of $I$, and $h:R\longrightarrow R$ is the unique $K$-algebra homomorphism 
satisfying $h(x_i)=x_i^{w_i}$  for $i=1,\ldots,n.$
}
\end{definition}

To illustrate this construction, consider $R=K[x_1,x_2,x_3,x_4]$ and the monomial ideal
 $I=(x_1^3x_2^2x_4,\;x_1x_3^2,\;x_2x_3x_4^2,\;x_3^3x_4).$ Suppose that $W$ is the weight on $R$ determined by
\[
W(x_1)=2,\qquad
W(x_2)=3,\qquad
W(x_3)=1,\qquad
W(x_4)=4.
\]
The homomorphism $h$ associated with this weight therefore satisfies
\[
h(x_1)=x_1^2,\qquad
h(x_2)=x_2^3,\qquad
h(x_3)=x_3,\qquad
h(x_4)=x_4^4.
\]
Applying $h$ to each minimal generator of $I$ gives
\[
h(x_1^3x_2^2x_4)=x_1^6x_2^6x_4^4,\; h(x_1x_3^2)=x_1^2x_3^2,\; h(x_2x_3x_4^2)=x_2^3x_3x_4^8, \; h(x_3^3x_4)
=x_3^3x_4^4.
\]
Consequently, the weighted ideal is  $I_W=(x_1^6x_2^6x_4^4,\;x_1^2x_3^2,\;x_2^3x_3x_4^8,\;x_3^3x_4^4).$


\medskip
To prove Lemma \ref{Lem.Fluctuation}, we will use the lemmas and proposition below, which we recall here for ease of reference.

\begin{lemma} (\cite [Lemma 1.1]{BH}) \label{Lem.Bayati.Expansion}
Let $I$ and $J$ be monomial ideals in a polynomial ring $S$. Then
\begin{itemize}
\item[(i)] $f \in I^*$ if and only if $\pi(f)\in I$, for all $f\in S^*$;
\item[(ii)] $(I + J)^* = I^* + J^*$;
\item[(iii)] $(IJ)^* = I^*J^*$;
\item[(iv)] $(I \cap J)^* = I^*\cap J^*$;
\item[(v)] $(I : J)^* = (I^* : J^*)$;
\item[(vi)] $\sqrt{I^*} = (\sqrt{I})^*$;
\item[(vii)] If the monomial ideal $Q$ is $\mathfrak{p}$-primary, then $Q^*$ is
$\mathfrak{p}^*$-primary.
\end{itemize}
\end{lemma}

\begin{proposition} (\cite [Proposition 1.2]{BH}) \label{Pro.Bayati.Expansion}
Let $I$ be a monomial ideal, and consider an (irredundant) primary
decomposition $I = Q_1\cap\cdots \cap Q_m$ of $I$. Then 
$I^*={Q^*_1} \cap \cdots \cap {Q^*_m}$ is an (irredundant) primary decomposition of $I^*$.
In particular, $\mathrm{Ass}(S^*/I^*) = \{\mathfrak{p}^* : \mathfrak{p} \in \mathrm{Ass}(S/I)\}.$
\end{proposition}

\begin{lemma}\label{LEM. Weighted} (\cite[Lemma 3.5]{SN})
Let $I$ and $J$ be two monomial ideals of a polynomial ring $R=K[x_1, \ldots, x_n]$, and $W$ a weight over $R$. Then the following statements hold. 
\begin {itemize}
\item[(i)] $(I+J)_W= I_W + J_W$;
\item[(ii)] $(IJ)_W= I_W J_W$;
\item[(iii)] $(I\cap J)_W= I_W \cap J_W$;
\item[(iv)] $(I :_RJ)_W= (I_W :_R J_W)$.
\end{itemize}
\end{lemma}

\begin{lemma}\label{ASS-Weighted} (\cite[Lemma 3.9]{SN})
Let $I$  be a monomial ideal in a polynomial ring $R=K[x_1, \ldots, x_n]$, and $W$ a weight over $R$. Then 
$\mathrm{Ass}_R(R/I_W)=\mathrm{Ass}_R(R/I).$
\end{lemma}


We are now ready to prove Lemma \ref{Lem.Fluctuation}, which plays a crucial role in the proof of Theorem \ref{Th.Fluctuation}.

\begin{lemma}\label{Lem.Fluctuation} 
Let  $I \subset R=K[x_1, \ldots, x_n]$ be a monomial ideal.  Then the following statements hold.
\begin{itemize}
\item[(1)] $I$ has the  phenomenon  of  fluctuation in associated primes of powers   if and only if $I^*$ has the  phenomenon  of  fluctuation in
associated primes of powers,   where $I^*$ denotes the  expansion of $I$.  
\item[(2)] $I$  has the  phenomenon  of  fluctuation in associated primes of powers   if and only if $I_W$ has the  phenomenon  of  fluctuation in
associated primes of powers,  where $I_W$ denotes the  weighted ideal.    
\end{itemize}
\end{lemma}

\begin{proof}
 (1)     First, assume that $I$ has the  phenomenon  of  fluctuation in associated primes of powers.  This implies that
    there exist  a monomial prime ideal $\mathfrak{p}$ in $R$, and  positive integers $a<b<c$ satisfying  at least one of the following cases:
\begin{itemize}
\item[(i)] $\mathfrak{p} \in \mathrm{Ass}(R/I^a)$, $\mathfrak{p} \notin \mathrm{Ass}(R/I^b)$, and $\mathfrak{p} \in \mathrm{Ass}(R/I^c)$. 
\item[(ii)] $\mathfrak{p} \notin \mathrm{Ass}(R/I^a)$, $\mathfrak{p} \in \mathrm{Ass}(R/I^b)$, and $\mathfrak{p} \notin \mathrm{Ass}(R/I^c)$.
\end{itemize}
Without loss of generality, assume that case (i) holds, as case (ii) follows similarly. From  Proposition \ref{Pro.Bayati.Expansion}, we can 
deduce that  
    $\mathfrak{p}^* \in \mathrm{Ass}(R^*/(I^a)^*)$, $\mathfrak{p}^* \notin \mathrm{Ass}(R^*/(I^b)^*)$, 
    and $\mathfrak{p}^* \in \mathrm{Ass}(R^*/(I^c)^*)$.   It follows from Lemma \ref{Lem.Bayati.Expansion}(iii) that 
$(I^{a})^*=(I^*)^{a}$, $(I^{b})^*=(I^*)^{b}$, and $(I^{c})^*=(I^*)^{c}$.  This gives that 
$\mathfrak{p}^* \in \mathrm{Ass}(R^*/(I^*)^a)$, $\mathfrak{p}^* \notin \mathrm{Ass}(R^*/(I^*)^b)$, 
    and $\mathfrak{p}^* \in \mathrm{Ass}(R^*/(I^*)^c)$. This means that  $I^*$  has the  phenomenon  of  fluctuation in associated primes of powers. 
The converse follows by a similar argument.

  (2)     The desired claim  follows by adapting the proof of (1) and using Lemmas \ref{LEM. Weighted} and \ref{ASS-Weighted} in a similar manner.
 \end{proof}


 With the preceding results in place, we now formulate and prove the following theorem, which constitutes the first application of 
 Example \ref{Counterexample.1}.

\begin{theorem} \label{Th.Fluctuation}
Let $R=K[x_1, \ldots, x_n]$ with $n\geq 3$ be  a polynomial ring in $n$  variables  with coefficients in a  field $K$. 
Then there exist infinitely many  monomial ideals  in $R$  that  have  the  phenomenon  of  fluctuation in  associated primes of powers.
\end{theorem}

\begin{proof}
Let $L$ be a monomial ideal in $S=K[x,y,z]$ exhibiting the phenomenon of fluctuation in the associated primes of its powers
 (see Example \ref{Counterexample.1}, which guarantees the existence of such a monomial ideal). 
 Set    $$\mathfrak{p}_1:=(x_{i_1}, \ldots, x_{i_\lambda}), \quad  \mathfrak{p}_2:=(x_{i_{\lambda+1}}, \ldots, x_{i_\theta}) \quad  
  \text{and} \quad  \mathfrak{p}_3:=(x_{i_{\theta+1}}, \ldots, x_{i_\rho}),$$   such that 
   $\bigcup_{i=1}^3\mathrm{supp}(\mathfrak{p}_i)=\{x_1, \ldots, x_n\}$ and   $\mathrm{supp}(\mathfrak{p}_i) \cap    \mathrm{supp}(\mathfrak{p}_j)=\emptyset$  for all $1\leq i < j\leq 3$.  We immediately deduce 
   from  Lemma  \ref{Lem.Fluctuation}(i) that  $L^*$ in $R=K[x_1, \ldots, x_n]$  has the  phenomenon  of  fluctuation in associated primes
   of powers.  Now, put    $I:=(L^*)_W$ with $W(x_i)=\alpha_i$ such that $\alpha_i \geq 1$ for all $i=1, \ldots, n$.  It follows from  Lemma  \ref{Lem.Fluctuation}(ii)  that   $I$  has  the  phenomenon  of  fluctuation in  associated primes of powers. 
   This permits us to conclude that    there exist infinitely many  monomial ideals  in     $R=K[x_1, \ldots, x_n]$    that  have  the  phenomenon  of
  fluctuation in the associated primes of  powers, as claimed.   
\end{proof}


  \subsection{Nearly normally torsion-free monomial ideals and the persistence property} 
  
   For the second application of Example \ref{Counterexample.1}, we first recall some definitions that will be used throughout this subsection.

\begin{definition}(\cite[Definition 2.1]{Claudia}) \label{Def.NNTF}
{\em
Let $I \subset R = K[x_1,\dots,x_n]$ be a monomial ideal. Then $I$ is called \textit{nearly normally torsion-free} if  there exist  a monomial prime ideal $\mathfrak{p}$  and an integer $m \ge 1$ such that  the following statements hold:
\begin{itemize}
\item[(i)]  $\operatorname{Ass}(R/I^t) = \operatorname{Min}(I)$ for all $1 \le t \le m$,
\item[(ii)]   $\operatorname{Ass}(R/I^t) \subseteq \operatorname{Min}(I) \cup \{\mathfrak{p}\}$ for all $t \ge m+1$.
\end{itemize}
}
\end{definition}

Examples  \ref{Ex.Ass.NandP}  and \ref{Ex.Ass.NandP.3}  exhibit classes of monomial ideals satisfying Definition \ref{Def.NNTF}. 
Particularly,  the maximal ideal may occur in certain classes of monomial ideals that are nearly normally
 torsion-free, such as dominating ideals (see \cite[Theorem 4.4]{NBR}  and \cite[Theorem 3.9]{NQBM}) and   $t$-spread
  monomial ideals (refer to \cite[Lemma 5.15]{NQKR}); see also  \cite[Theorem 4.2]{NT.2}  for further information.

  \begin{definition}
    \em{Let $I$ be an ideal in  a commutative Noetherian ring  $R$.
 We  say that  $I$   has the \textit{persistence property} if
 $\mathrm{Ass}(R/I^t) \subseteq \mathrm{Ass}(R/I^{t+1})$ for all $t\geq 1$.
   }
  \end{definition}

The following open question was previously posed in \cite[Question 2.10]{Nasernejad-1}.

\begin{question}
Let $I$ be a square-free monomial ideal in a polynomial ring
$R=K[x_1,\ldots,x_n]$ over a field $K$. Suppose that $I$ is nearly normally
torsion-free. Can we conclude that $I$ has the persistence property or the
strong persistence property?     
\end{question}

  In Proposition \ref{Pro.NNTF}, we provide a negative answer to this question in the monomial case. To achieve this,   we need the following auxiliary lemma, whose proof is straightforward and left to the reader.
  
  \begin{lemma}\label{Pro.NNTF.Expansion.Weighted}
Let $I$ be a monomial ideal in $R=K[x_1,\ldots,x_n]$.   Then
\begin{itemize}
\item[(1)]   $I$ is nearly normally torsion-free if and only if  $I^*$ is nearly normally torsion-free.
\item[(2)]  $I$ is nearly normally torsion-free if and only if   $I_W$ is  nearly normally torsion-free.
\end{itemize}
\end{lemma}

 \begin{remark}\label{Rem.NNTF}
 {\em 
 Let $I\subset R=K[x_1, \ldots, x_n]$ be a monomial ideal. According to \cite[Lemma 2.9]{N3}, $I$ has the persistence property if and only if $I^*$ 
 has the persistence property. In addition, based on    Lemmas \ref{LEM. Weighted} and \ref{ASS-Weighted}, it is easy to show that $I$ has the persistence property if and only if $I_W$   has the persistence property. 
 }
 \end{remark}

 \begin{proposition}\label{Pro.NNTF}
 Let $R=K[x_1,\ldots,x_n]$, with $n\geq 3$, be a polynomial ring in
$n$ variables over a field $K$. Then there exist infinitely many monomial
ideals in $R$ that are nearly normally torsion-free but do not satisfy the
persistence property.
 \end{proposition}
 
 \begin{proof}
  Let $I\subset S=K[x,y,z]$ be the monomial ideal defined in Example \ref{Counterexample.1}(1). It is routine to check that $I$ is nearly normally torsion-free. 
 Set    $$\mathfrak{p}_1:=(x_{i_1}, \ldots, x_{i_\lambda}), \quad  \mathfrak{p}_2:=(x_{i_{\lambda+1}}, \ldots, x_{i_\theta}) \quad  
  \text{and} \quad  \mathfrak{p}_3:=(x_{i_{\theta+1}}, \ldots, x_{i_\rho}),$$   such that 
   $\bigcup_{i=1}^3\mathrm{supp}(\mathfrak{p}_i)=\{x_1, \ldots, x_n\}$ and   $\mathrm{supp}(\mathfrak{p}_i) \cap    \mathrm{supp}(\mathfrak{p}_j)=\emptyset$  for all $1\leq i < j\leq 3$.  It follows from Lemma  \ref{Pro.NNTF.Expansion.Weighted}(i) that  $I^*$ in $R=K[x_1, \ldots, x_n]$ is nearly normally torsion-free. Let   $J:=(I^*)_W$ with $W(x_i)=\alpha_i$ such that $\alpha_i \geq 1$ for all $i=1, \ldots, n$.  
   One can derive from Lemma  \ref{Pro.NNTF.Expansion.Weighted}(ii) that   $J$ is nearly normally torsion-free.  
   Since $I$ does not satisfy the persistence property, Remark \ref{Rem.NNTF} yields that $I^*$ does not satisfy the persistence property. 
  Once again, Remark \ref{Rem.NNTF}  shows that  $J=(I^*)_W$ does not satisfy the persistence property as well.   This implies that   there exist infinitely many monomial
ideals in $R$ that are nearly normally torsion-free but do not satisfy the persistence property.
 \end{proof}

 
 \subsection{Co-nearly normally torsion-free monomial ideals and the  copersistence property} 
 
  We turn to the third application of Example \ref{Counterexample.1} by first recalling a few definitions that will be necessary  in the sequel.

   \begin{definition}\label{Def.co-NNTF}
\em{
Let $I \subset R = K[x_1,\dots,x_n]$ be a  monomial ideal. We say that  $I$ is \textit{co-nearly normally torsion-free}
 if there exist  a monomial prime ideal $\mathfrak{p}$ 
and an integer $m \geq 1$ such that  the following statements hold:
\begin{enumerate}
\item[(i)]  $\operatorname{Ass}(R/I^t) = \operatorname{Min}(I) \cup \{\mathfrak{p}\}$ for all $1 \leq t \leq m$,
\item[(ii)]   $\operatorname{Ass}(R/I^t) \subseteq \operatorname{Min}(I) \cup \{\mathfrak{p}\}$ for all $t \geq m+1$.
\end{enumerate}
}
\end{definition}

Examples  \ref{Ex.Ass.PandN},  \ref{Ex.Ass.PandN.2}, and \ref{Ex.Ass.PandN.3} describe classes of monomial ideals satisfying 
Definition \ref{Def.co-NNTF}.

 \begin{definition} (\cite{NT.3})
    \em{Let $I$ be an ideal in  a commutative Noetherian ring  $R$.
      We say   $I$  has the \textit{copersistence property} if 
$\mathrm{Ass}_R(R/I^t) \supseteq  \mathrm{Ass}_R(R/I^{t+1})$ for all $t\geq 1$.   
   }
  \end{definition}
  
  In view of the above definition, it is natural to ask the following question.
  
  \begin{question}
  Does every co-nearly normally torsion-free monomial ideal satisfy the copersistence property?
    \end{question}

  Proposition \ref{Pro.co-NNTF} shows that, in general, the answer to this question is negative. To establish Proposition  \ref{Pro.co-NNTF}, 
  we have to  use the following auxiliary lemma, whose proof is straightforward and left to the reader.
  
   \begin{lemma}\label{Pro.co-NNTF.Expansion.Weighted}
Let $I$ be a monomial ideal in $R=K[x_1,\ldots,x_n]$.   Then
\begin{itemize}
\item[(1)]   $I$ is co-nearly normally torsion-free if and only if  $I^*$ is co-nearly normally torsion-free.
\item[(2)]  $I$ is co-nearly normally torsion-free if and only if   $I_W$ is  co-nearly normally torsion-free.
\end{itemize}
\end{lemma}
  
  \begin{remark}\label{Rem.CONNTF}
  {\em
 Let $I\subset R=K[x_1, \ldots, x_n]$ be a monomial ideal. It follows from \cite[Theorem 3.10]{NT.3} that  $I$ has the copersistence property if and only if $I^*$  has the copersistence property. Moreover,  one can deduce from \cite[Theorem 3.14]{NT.3} that  $I$ has the copersistence property if and only if $I_W$   has the copersistence property. 
 }
 \end{remark}

   \begin{proposition} \label{Pro.co-NNTF}
 Let $R=K[x_1,\ldots,x_n]$, with $n\geq 3$, be a polynomial ring in
$n$ variables over a field $K$. Then there exist infinitely many monomial
ideals in $R$ that are co-nearly normally torsion-free but do not satisfy the
copersistence property.
 \end{proposition}
  
  \begin{proof}
  The assertion follows readily  from Lemma \ref{Pro.co-NNTF.Expansion.Weighted} and Remark \ref{Rem.CONNTF}
   by repeating the proof of Proposition \ref{Pro.NNTF}. 
  \end{proof}
  
  
  \bigskip
  We conclude this paper by giving a negative answer to the following question.

 \begin{question}  
Let $I\subset R=K[x_1, \ldots, x_n]$ be a monomial ideal such that 
 $\mathrm{Ass}(R/I)=\{\mathfrak{p} :  2\leq |\mathrm{supp}(\mathfrak{p})| \leq n\}$. Then  
$\mathrm{Ass}(R/I) = \mathrm{Ass}(R/I^t)$ for all  $t\geq 1.$
\end{question}

The answer to this question is negative in general. To see such a counterexample, consider the following monomial ideal in $R=K[x,y,z]$
$$I =(y^5z,\,xy^4z,\,x^2y^3z^2,\,x^3y^2,\,x^4y,\,x^5z).$$
 Set $\mathfrak{m}:=(x,y,z)$.  Using  \textit{Macaulay2} \cite{GS}, we obtain 
 \[
\begin{aligned}
\operatorname{Ass}(R/I)&=\{(x,y),(x,z), (y,z), (x,y,z)\}=\mathrm{Min}(I)\cup \{\mathfrak{m}\},\\
\operatorname{Ass}(R/I^2)&=\{(x,y),(x,z), (y,z)\}=\mathrm{Min}(I),\\
\operatorname{Ass}(R/I^3)&=\{(x,y),(x,z), (y,z)\}=\mathrm{Min}(I).
\end{aligned}
\]
In particular, the ideal $I$  is co-nearly normally torsion-free. However
 $$\mathrm{Ass}(R/I)=\{\mathfrak{p} :  2\leq |\mathrm{supp}(\mathfrak{p})| \leq 3\}=\{(x,y),(x,z), (y,z), (x,y,z)\},$$ whereas
$\operatorname{Ass}(R/I) \neq \operatorname{Ass}(R/I^2)$.  Thus, the monomial ideal  $I$ provides the required counterexample.


 \section{Conclusion and Outlook}

In this paper, we investigated the presence of the maximal ideal in the set of associated primes of powers of monomial ideals, as well as the phenomenon of fluctuation. By employing depth-theoretic and combinatorial methods, we provided criteria for the occurrence of the maximal ideal and constructed infinite families of monomial ideals exhibiting fluctuation in their associated primes.

Regarding the general status of the fluctuation problem for monomial ideals:
\begin{itemize}
    \item \textbf{General Monomial Ideals:} The problem of fluctuation is completely settled. For every $n \ge 3$, there exist infinitely many general monomial ideals in $K[x_1, \dots, x_n]$ whose sets of associated primes exhibit fluctuation.
    \item \textbf{Square-Free Monomial Ideals:} The situation is nearly completely resolved:
    \begin{enumerate}
        \item For $n \le 5$, no square-free monomial ideal exhibits fluctuation due to the strong persistence property.
        \item For $n \ge 7$, the existence of a single square-free counterexample in seven variables allows us to construct  square-free monomial ideals
        counterexamples in $K[x_1, \dots, x_n]$ for any $n \ge 7$ using the expansion  technique.
        \item Consequently, the \textbf{only remaining open case} for square-free monomial ideals is $n = 6$.
    \end{enumerate}
\end{itemize}

\subsection*{Future Directions}

The primary open question arising from this line of inquiry is to settle the final boundary case:

\begin{question}
 Does there exist a square-free monomial ideal in the polynomial ring $R = K[x_1, x_2, x_3, x_4, x_5, x_6]$ whose powers exhibit fluctuation in their sets of associated primes?
\end{question}

Resolving the $n = 6$ case will fully complete the classification of the fluctuation phenomenon for square-free monomial ideals across all polynomial rings.

 The constructions obtained in this work naturally lead to the following  open questions concerning the possible fluctuation of associated primes of powers of monomial ideals. They ask whether the two types of fluctuation described below can occur for arbitrary prescribed integers $1\leq\alpha<\beta.$

\begin{question}
Given $1\leq \alpha < \beta$, does there exist a monomial ideal $I$ in $R=K[x,y,z]$ such that 
\[
\mathrm{Ass}(R/I^t) =
\begin{dcases}
\mathrm{Min}(I) & \text{if} \hspace{.5cm} 1\leq t\leq \alpha\\
\mathrm{Min}(I) \cup \{(x,y,z)\} & \text{if} \hspace{.5cm} \alpha+1 \leq t \leq \beta \\
\mathrm{Min}(I) & \text{if} \hspace{.5cm} t\geq \beta+1.\\
\end{dcases}
\]
\end{question}

\begin{question}
Given $1\leq \alpha < \beta$, does there exist a monomial ideal $I$ in $R=K[x,y,z]$ such that 
\[
\mathrm{Ass}(R/I^t) =
\begin{dcases}
\mathrm{Min}(I) \cup \{(x,y,z)\} & \text{if}\hspace{.5cm} 1\leq t\leq \alpha\\
\mathrm{Min}(I) & \text{if} \hspace{.5cm} \alpha+1 \leq t \leq \beta \\
\mathrm{Min}(I) \cup \{(x,y,z)\} & \text{if} \hspace{.5cm} t\geq \beta+1.\\
\end{dcases}
\]
\end{question}


\vspace{1.2cm}
\noindent\textbf{Acknowledgements}\\
Some  part of this paper was prepared when the second author, Mehrdad Nasernejad,   visited Alpen-Adria-Universit\"at Klagenfurt  in 2024; in particular, he  would like to thank 
Alpen-Adria-Universit\"at Klagenfurt for its hospitality and support during his stay.
 
\bigskip
\noindent\textbf{Funding Declaration}\\
  The authors  received no financial support for the research, authorship, or publication of this article.

 \bigskip
\noindent\textbf{Data availability statement}\\
This manuscript has no associated data.

\bigskip
\noindent\textbf{Disclosure statement}\\
The authors declare no financial or non-financial competing interests.

\bigskip
\noindent\textbf{Orcid}\\
Mircea Cimpoea\c{s}:~~${0000000277749883}$\\
Mehrdad Nasernejad:~~${0000000310731934}$\\
Ayesha Asloob Qureshi:~~${0000000234002069}$



\end{document}